\pdfoutput=1    % arXiv recommends this as the first line; see http://arxiv.org/help/submit_tex

\documentclass[11pt,a4paper]{article} 

\usepackage[OT2,T1]{fontenc}
\usepackage[utf8]{inputenc}
\usepackage[russian,ngerman,english]{babel}
\usepackage{amsmath}
\usepackage{amssymb,verbatim}
\usepackage{amsfonts}
\usepackage{amsthm}
\usepackage{pifont}
\usepackage{graphicx}
\usepackage{datetime}
\usepackage{scrdate}
\usepackage{enumitem}
\usepackage{geometry}
\usepackage{float}
\usepackage{bigints}
\usepackage{multirow}
\usepackage{bbm} % fuer die Indikatorfunktions-1
\usepackage{dsfont} % fuer die Indikatorfunktions-1
\usepackage{cases}
\usepackage[labelfont=bf]{caption}
\usepackage{tabularx}
\usepackage{hyperref}
\usepackage[
  backend=biber,
  style=authoryear-icomp, % numeric, icomp bewirkt Zitat mit ebd.
  natbib=true,
  uniquename=false
  ]{biblatex}
\usepackage{parskip}  % skip a line between paragraphs
\usepackage{sidecap} % für Bezeichnung neben Abb.
\usepackage{esvect}

\hypersetup{colorlinks=true, linkcolor=blue, citecolor=blue, pdfstartview=FitH, linktocpage=false}

\settimeformat{hhmmsstime}

\newtheorem{thm}{Theorem}%[section]
\newtheorem{lem}{Lemma}
\newtheorem{cor}{Corollary}

\theoremstyle{definition}

\newtheorem{remark}{Remark}

\newtheorem{examples}{Examples}

\addto\captionsenglish{

}
\makeatletter 
\newcommand\mynobreakpar{\par\nobreak\@afterheading} 
\makeatother

\newcommand{\R}{\mathbb{R}}

\newcommand{\beq}{\begin{equation*}}
\newcommand{\eeq}{\end{equation*}}
\newcommand{\beqn}{\begin{equation}}
\newcommand{\eeqn}{\end{equation}}
\newcommand{\dd}{\mathrm{d}}

\newcommand{\ds}{\displaystyle}
\newcommand{\ph}{\varphi}
\newcommand{\rh}{\varrho}
\newcommand{\D}{\varDelta}
\newcommand{\La}{\varLambda}
\newcommand{\la}{\lambda}

\newcommand{\db}{\displaybreak[0]}
\newcommand{\Po}{\mathcal{P}_{n,\;\!r}}

\newcommand{\Cr}{\mathcal{C}_r}
\newcommand{\E}{\mathbb{E}}
\newcommand{\Var}{\mathrm{Var}}
\newcommand{\ii}{\mathrm{i}}
\newcommand{\ee}{\mathrm{e}}
\newcommand{\FL}{F_{n,\:\!r}}
\newcommand{\FD}{G_{n,\:\!r}}
\newcommand{\fL}{f_{n,\:\!r}}
\newcommand{\fD}{g_{n,\:\!r}}
\newcommand{\Ga}{\varGamma}
\newcommand{\Ht}{\widetilde{H}}

\newcommand{\p}{\partial}
\newcommand{\AP}{A(\Po)}
\newcommand{\LP}{L(\p\:\!\Po)}
\newcommand{\psit}{\widetilde{\psi}}
\newcommand{\sigmat}{\widetilde{\sigma}}
\newcommand{\taut}{\widetilde{\tau}}
\newcommand{\g}{\mathfrak{g}}
\DeclareMathOperator{\arcosh}{arcosh}

\begin{document}
% !TeX root = MDRP_0.tex

\title{The moments of the distance between two random points\\ in a regular polygon}
\author{Uwe Bäsel}
\date{\today,\;\currenttime}
\date{} 
\maketitle
\thispagestyle{empty}
\begin{abstract}
\noindent
In this paper, we derive formulas for the analytical calculation of the moments of the distance between two uniformly and independently distributed random points in an $n$-sided regular polygon. A number of closed form expressions is provided, e.g. the expected distances for $n = 3,4,5,6,8,10,12$, where the results for $n = 5,8,10,12$ are new to the best of the author's knowledge.
Applying results of Voss, remarkably short formulas for the second and the fourth moments are derived.\\[0.2cm]
\\[0.2cm]
\end{abstract}

\section{Introduction}

Let $\Po$, $n = 3,4,\ldots$, be the regular polygon with $n$ sides and circumscribed circle with radius~$r$.
The distance between two points chosen independently and uniformly from $\Po$ is a random variable, which we denote by $\D_{n,\:\!r}$.
The moments of $\D_{n,\:\!r}$ are given by
\beq
  M_m(\Po)
:= \E[\D_{n,\:\!r}^m] 
= \int_0^d x^m\, \dd\FD(x)\,,
\eeq
where $G_{n,\:\!r}$ is the cumulative distribution function of $\D_{n,\:\!r}$, and $d$ is the diameter of $\Po$.
The first moment $M_1(\Po)$ is the {\em mean distance} between two (random) points in $\Po$.

\textcite{Czuber1884} found the mean values $M_1(\mathcal{P}_{3,\:\!r})$ (p.\ 206) and $M_1(\mathcal{P}_{4,\:\!r})$ (pp.\ 202-204, Problem XV) for the equilateral triangle and the square, respectively, in terms of the side length $a = 2r\sin(\pi/n)$. (Czuber even found the mean distance in a rectangle. Concerning equilateral triangle, square and rectangle see also \textcite[p.\ 49]{Santalo}.)
\textcite{Ghosh} derived the distance distribution for a rectangle.
\textcite{Sulanke} found the distribution function $G_{3,\:\!r}$ for the equilateral triangle in terms of the side length (see also \textcite{Duma&Stoka08}).    
The density function $\fD$ and the distribution function $\FD$ for any regular  polygon were obtained by \textcite{Baesel:Random_chords}.
An example for the simulation of point distances is shown in Fig.\ \ref{Abb:Simulation}.

\textcite{Hammersley} derived the distribution of the distance in a hypersphere.
\textcite{Bailey&Borwein&Crandall:Box} studied the mean distance in the unit $n$-cube giving closed form expressions for the cases $n=1,\ldots,4$ (see also \textcite{Borwein&Crandall}, especially Example 14).
The density function of the distance between two points in a three-dimensional box was found by \textcite{Philip:Box}.
Furthermore, \textcite{Philip} found the density function for a 4-cube and a 5-cube. 

Although the density function and the distribution function provide much more information than the mean value, there is a constant interest in mean values; see e.\,g.\ \textcite{Dunbar}, and \textcite{Burgstaller&Pillichshammer}.
Recently, \textcite{Bonnet&Gusakova&Thaele&Zaporozhets} obtained sharp inequalities for the mean distance in convex bodies.

\begin{SCfigure}[][h]
  \includegraphics[width=0.5\textwidth]{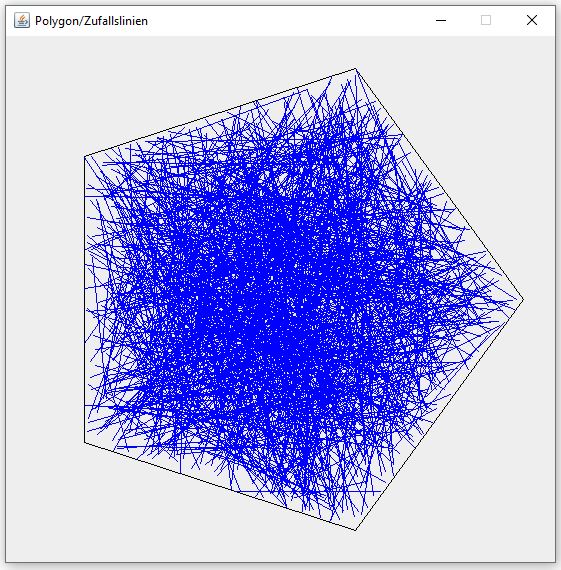}
  \caption{Simulation of 1000 line segments each as a connection of two random points in $\mathcal{P}_{5,\:\!1}$ using PolygonV111.jar \autocite{Achtelstaedter&Huebener_Polygon}\\[0.2cm]
  Here the empirical values of $M_1(\mathcal{P}_{5,\:\!1})$ and the variance $\Var[\D_{5,\:\!1}]$ are 0.7911 (cf.\ Tab.\ \ref{Ta:Moments-P_5}) and 0.1316, respectively. Range: 1.6795}
  \label{Abb:Simulation}
\end{SCfigure}

Without loss of generality one may assume that the centroid of $\Po$ is located at the origin $O$ of a cartesian $x,y$-coordinate system, and the vertices of $\Po$ are given by
\beqn \label{Eq:X_k}
  X_k = x_k + \ii y_k = r\ee^{2\ii k\alpha}\,,\quad
  k = 0,1,\dots,n-1,
\eeqn
with angle
\beq
  \alpha
= \frac{1}{2}\:\measuredangle\left(\vv{OX_0},\vv{OX_1}\right)
= \frac{\pi}{n}\,.
\eeq 
The distance $\ell_k$ between vertex $X_0$ and vertex $X_k$ is for $k = 1,2,\ldots,n-1$ given by
\beqn \label{Eq:ell_k}
  \ell_k
= r\,\sqrt{\strut\left(\cos(2k\alpha)-1\right)^2 + \sin^2(2k\alpha)}
= 2r\sin(k\alpha)\,.  
\eeqn 
Clearly, $\ell_1 = 2r\sin\alpha$ is the side length, and, denoting by $\lfloor\cdot\rfloor$ the integer part of $\,\cdot\,$,
\beqn \label{Eq:d}
  d :=\ell_{K+1} \quad\mbox{with}\quad K := \lfloor n/2 \rfloor - 1
\eeqn
is the diameter of $\Po$.  
The length $\LP$ of the boundary $\p\:\!\Po$ and the area $\AP$ of $\Po$ are respectively
\beqn \label{Eq:L_and_A}
  \LP
= n\,\ell_1
= 2nr\sin\alpha\qquad\mbox{and}\qquad  
  \AP
= n\,\frac{\ell_1}{2}\,r\cos\alpha 
= \frac{1}{2}\,n r^2\sin(2\alpha)\,.
\eeqn
For $\Po$ with odd $n$ we denote by $\la$ the distance between the vertex $X_0$ and the side $\overline{X_{K+1}X_{K+2}}$,
\beqn \label{Eq:lambda}
  \la
= r + r\cos\alpha
= 2r\cos^2(\alpha/2)\,. 
\eeqn
We will also use \eqref{Eq:lambda} for $\Po$ with even $n$.
Here $\lambda$ is the sum of $r = \left|\overline{OX_0}\right|$ and the distance between $O$ and side $\overline{X_KX_{K+1}}$ (and also $\overline{X_{K+1}X_{K+2}}$\,).
For $\Po$, $n = 3,4,5,\ldots$, we have
\beq
  \ell_K < \la < \ell_{K+1} = d\,.
\eeq  

The paper is organized as follows:

\begin{itemize}[leftmargin=0.5cm]
\setlength{\itemsep}{-2pt}
%%%%%%%%%%
\item In Section \ref{Sec:CLD} we provide a simple and clear form of the chord length distribution function $\FL$ derived in \textcite{Baesel:Random_chords}.
This simple form is completely given in Theorem \ref{Thm:CLD}.
%%%%%%%%%%
\item For the point distance density function $\fD$ and the moments $M_m(\Po)$ required antiderivatives are provided in Section \ref{Sec:Antiderivatives}.  
%%%%%%%%%%
\item From $\FL$ we conclude in Section \ref{Sec:PDF} a simple and clear form of the point distance density function $\fD$ derived in \textcite{Baesel:Random_chords}.
%%%%%%%%%%
\item Formulas for analytically calculating the distance moments $M_m(\Po)$, $m = -1,0,1,2,\ldots$, are derived from $\FL$ in Section \ref{Sec:Moments_distance}.
%%%%%%%%%%
\item Closed form expressions for the mean distances $M_1(\Po)$ and a number of further moments are given for $n = 3,4,5,6,8,10,12$ in Section \ref{Sec:Closed_form_expressions}.
According to the author's knowledge, cases $n = 5,8,10,12$ are new results.
%%%%%%%%%%
\item High precision values with respectively 75 and 76 digits for $M_1(\Po)/r$, $n = 3,4,\ldots,30$, and $M_m(\mathcal{P}_{5,\:\!r})/r^m$, $m = -1,0,1,2,\ldots,10$, are given in Section \ref{Sec:Numerical_values}.
%%%%%%%%%%
\item Using a result from \textcite{Czuber1884}, in Section \ref{Sec:Second_moment} a short formula for the second moments $M_2(\Po)$ is derived.
This makes it easier to calculate the variances $\Var[\D_{n,\:\!r}]$.
Especially in all cases, where a (short) closed form expression for the mean distance $M_1(\Po)$ is found, one has therefore also a closed form expression for the variance.
%%%%%%%%%%
\item Applying a result from \textcite{Voss82}, we obtain in Section \ref{Sec:Fourth_moment} a short formula for the fourth moments $M_4(\Po)$.
For this purpose it is used that the value of the second area moment of $\Po$ with respect to an axis through the centroid of $\Po$ is independent from the direction of this axis.
%%%%%%%%%%
\item Since clearly the moments for the circle $\Cr$ with radius $r$ are the limit values of the moments for $\Po$ as $n$ tends to infinity, $\lim_{n\rightarrow\infty}M_m(\Po) = M_m(\Cr)$, in Section \ref{Sec:Circle} the already known moments are derived as values for comparison.
(Comparisons have already been done in previous sections.)
%%%%%%%%%%
\item Finally, we point out the possibility of the simple derivation of chord power integrals for $\Po$ from the present results (Section \ref{Sec:CPI}). 
\end{itemize}
% !TeX root = MDRP_0.tex

\section{The chord length distribution function}
\label{Sec:CLD}

A straight line $\g$ in the plane is in Hesse normal form determined by the angle $\ph$, $0\le\ph<2\pi$, between the positive $x$-axis and the direction perpendicular to $\g$, and by its distance $p$, $0\le p<\infty$, from the origin~$O$,
\beq
  \g
= \g(p,\ph)
= \left\{(x,y)\in\R^2 \colon x\cos\ph + y\sin\ph = p\right\}
\eeq
(see \textcite[p.\ 2]{Santalo}).
We assume $p$ and $\ph$ to be independent random variables uniformly distributed in $[0,r]$ and $[0,2\pi)$, respectively, and consider only lines $\g$ with $\g\cap\Po\ne\emptyset$.
Such a line $\g$ produces a chord of length $\La_{n,\:\!r}$, $0 \le \La_{n,\:\!r} \le d = \ell_{K+1}$ (see \eqref{Eq:ell_k}, \eqref{Eq:d}).
We denote by $\FL$ the distribution function of the random variable $\La_{n,\:\!r}$, $\FL(x) = P(\La_{n,\,r} \le x)$.

\begin{thm} \label{Thm:CLD}
With
\begin{gather*}
  \alpha
= \pi/n\,,\qquad
  K = \lfloor n/2\rfloor - 1\,,\qquad
  \ell_k
= 2r\sin(k\alpha)\,,\qquad
  \la = 2r\cos^2(\alpha/2)\,,\\[0.1cm]
  h_k
= 2r\sin(k\alpha)\sin((k+1)\alpha)\,,\qquad
  p_k 
= \left\{\!\begin{array}{c@{\quad\mbox{if}\quad}l}
	0 & n = 3\,,\\[0.1cm]  
	\cot(2k\alpha) - \cot(2(k+1)\alpha) & n > 3\,,
  \end{array}\right.\\[0.1cm]  	
  q_k
= \tan(k\alpha) - \tan((k+1)\alpha)\,,\qquad
  s_k 
= \left\{\!\begin{array}{c@{\quad\mbox{if}\quad}l}
	1/2 & n = 3\,,\\[0.1cm]  
	k\,\alpha\,p_k & n > 3\,,
  \end{array}\right.\\[0.1cm]
  \sigma_{\mu,\:\!a}(x)
= x^\mu\arcsin\frac{a}{x}\,,\qquad
  \tau_{\mu,\:\!a}(x)
= x^\mu\,\sqrt{x^2-a^2}\,,
\end{gather*}
the chord length distribution function $\FL$ of the regular polygon $\Po$, $n = 3,4,\ldots$, is given by
\beq
  \FL(x) = P(\La_{n,\,r} \le x)
= \left\{
  \begin{array}{c@{\quad\mbox{if}\quad}l}
	0 & -\infty<x<\ell_0=0\,,\\[0.15cm]
	1 - H_k(x)/\ell_1  & \ell_k\leq x<\ell_{k+1} \;\;\mbox{for}\;\; k=0,\ldots,K,\\[0.15cm]
	1 & \ell_{K+1}\leq x<\infty
  \end{array}\right. 
\eeq
with
\beq
  H_k(x) = \left\{
  \begin{array}{ll}
	\ell_1 - \left[1+\alpha\left(\tan\alpha-\cot\alpha\right)\right]x/2 &
	%% if %% 
	\hspace{-5cm}\mbox{if}\quad k=0\,\wedge\,x<\la\,,\\[0.3cm]
	%%%%%%%%%%%%%%%%%%%%%%%%%%%%%%%%%%%%%%%%%%%%%%%%%%%%%%%%%%%%%%%%%%%%%%%%%%%%%%%%%%%%%%%%%%%%%%%%%%%%%%%%%%%%
	p_k\,\sigma_{1,\:\!h_k}(x)-\left[k\alpha\cot(2k\alpha)
	- (k+1)\alpha\cot(2(k+1)\alpha)\right]x\\[0.15cm]
	- \left(h_k\,p_k+2rq_k\cos\alpha\right)\tau_{-1,\:\!h_k}(x) &
	%% if %%
	\hspace{-5cm}\mbox{if \, $(n$ is even $\,\wedge\,$ $1 \le k \le K-1)$}\;\,\vee\\[0.15cm]
	& \hspace{-4.4cm}\mbox{$(n$ is odd $\,\wedge\,$ $1 \le k \le K$ $\,\wedge\,$ $x<\la)$}\,,\\[0.3cm]
	%%%%%%%%%%%%%%%%%%%%%%%%%%%%%%%%%%%%%%%%%%%%%%%%%%%%%%%%%%%%%%%%%%%%%%%%%%%%%%%%%%%%%%%%%%%%%%%%%%%%%%%%%%%%
	\left[1/2-k\alpha\cot(2k\alpha)\right]x
	+ \cot(2k\alpha)\,\sigma_{1,\:\!h_k}(x) - \left(h_k-2r\cos\alpha\right)h_k/x\\[0.15cm]
	- \left[h_k\cot(2k\alpha)+2r\cos\alpha\tan(k\alpha)\right]\tau_{-1,\:\!h_k}(x) &
	%% if &&
	\hspace{-5cm}\mbox{if}\quad \mbox{$n$ is even}\,\wedge\,k=K\,,\\[0.3cm]
	%%%%%%%%%%%%%%%%%%%%%%%%%%%%%%%%%%%%%%%%%%%%%%%%%%%%%%%%%%%%%%%%%%%%%%%%%%%%%%%%%%%%%%%%%%%%%%%%%%%%%%%%%%%%
	p_k\,\sigma_{1,\:\!h_k}(x) + 2\cot(2(k+1)\alpha)\,p_k\,\sigma_{1,\:\!\la}(x)
	- \left[(\pi-\alpha)\cot(2(k+1)\alpha)+s_k\right]x\\[0.15cm]
	- \left(h_k\,p_k+2rq_k\cos\alpha\right)\tau_{-1,\:\!h_k}(x)
	- 2\cos\alpha\,[2r\cos(\alpha/2)\sec((k+1)\alpha)\\[0.15cm]
	- \,\la\csc(2(k+1)\alpha)]\,\tau_{-1,\:\!\la}(x) &
	%% if %%	
	\hspace{-5cm}\mbox{if}\quad\mbox{$n$ is odd}\,\wedge\,k=K\,\wedge\,x\ge\la\,.		
  \end{array}\right.
\eeq
\end{thm}

\begin{proof}
This is the chord length distribution function from \textcite[Theorem 1]{Baesel:Random_chords}, simplified by inserting, rearranging and combining terms. 
Furthermore, the former case
\beq
  (\mbox{$n$ is even}\,\wedge\,k\in\{0,\ldots,K-1\})\,\vee\,(\mbox{$n$ is odd}\,\wedge\,x<\la)
\eeq
is split into two cases.
\end{proof}

Examples for graphs of distribution functions $\FL$ are shown in Fig.\ \ref{Abb:CLDs}.

\begin{SCfigure}[][h]
  \includegraphics[width=0.65\textwidth]{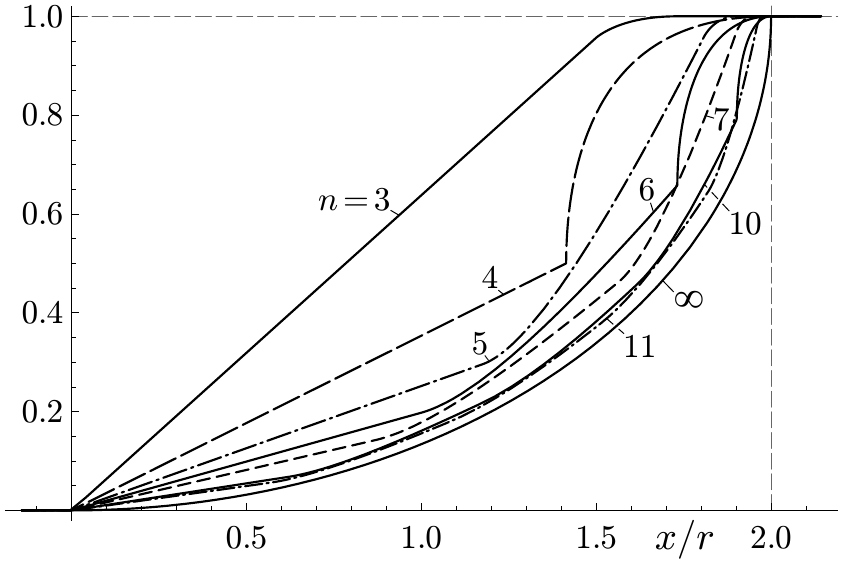}
  \caption{Graphs of distribution functions $\FL$, $n=3$, $4$, $5$, $6$, $7$, $10$, $11$, and limit distribution function (circle)}
  \label{Abb:CLDs}
\end{SCfigure}
% !TeX root = MDRP_0.tex

\section{Some antiderivatives}
\label{Sec:Antiderivatives}

\begin{lem} \label{Lem:psit}
For $\mu = -1,0,1,2,\ldots$, an antiderivative $\psit_{\mu,\:\!a}(x)$ of the function
\beq
  \psi_{\mu,\:\!a}
  \colon [a,\infty) \:\rightarrow\: \R\,,\quad
  x \:\mapsto\: \psi_{\mu,\:\!a}(x) := \frac{x^\mu}{\sqrt{x^2-a^2}}\,,\quad
  a > 0\,,  
\eeq
is given by
\beq
  \psit_{\mu,\:\!a}(x)
= \left\{
  \begin{array}{lcl}
    \ds{-\frac{1}{a}\arcsin\frac{a}{x}} & \mbox{if} & \mu = -1\,,\\[0.4cm]
    \ds{\ln\frac{x+\sqrt{x^2-a^2}}{a}} & \mbox{if} & \mu = 0\,,\\[0.4cm]
	\ds{\frac{x^{\mu-1}\,\sqrt{x^2-a^2}\;\gamma_{\mu,\:\!a}(x)}{\mu}} & \mbox{if} &
		\mu = 1,3,5,\ldots\,,\\[0.4cm]
	\ds{\frac{x^{\mu-1}\,\sqrt{x^2-a^2}\;\gamma_{\mu,\:\!a}(x)}{\mu}
	+ \frac{(\mu-1)!!\,a^\mu}{\mu!!}\,\ln\frac{x+\sqrt{x^2-a^2}}{a}} &
	\mbox{if} & \mu = 2,4,6,\ldots\,,
  \end{array}
  \right.	  
\eeq
where
\beq
  \gamma_{\mu,\:\!a}(x)
= 1 + \sum_{\nu=1}^{\lfloor\frac{\mu-1}{2}\rfloor}\left(\frac{a}{x}\right)^{2\nu}
	\prod_{j=1}^\nu\frac{\mu+1-2j}{\mu-2j}\,.  
\eeq
\end{lem}

\begin{proof}
In the case $\mu = -1$ we have
\beq
  \psit_{-1,\:\!a}(x)
= \int\frac{\dd x}{x\,\sqrt{x^2-a^2}} 
= \int\frac{\dd x}{x^2\,\sqrt{1-(a/x)^2}}\,. 
\eeq
The substitution $u = a/x$, $\dd u = -a\,\dd x/x^2$ gives
\beq
  \psit_{-1,\:\!a}(x)
= -\frac{1}{a}\int\frac{\dd u}{\sqrt{1-u^2}}
= -\frac{1}{a}\arcsin u
= -\frac{1}{a}\arcsin\frac{a}{x}\,.  
\eeq
For $\mu = 0,1,2,\ldots$, the substitution
\beq
  x = a\cosh u\,,\qquad \sqrt{x^2-a^2} = a\sinh u\,,\qquad \dd x = a\sinh u\,\dd u
\eeq
yields
\beq
  \psit_{\mu,\:\!a}(x)
= \int\frac{x^\mu}{\sqrt{x^2-a^2}}\,\dd x
= a^\mu\int\cosh^\mu u\,\dd u\,.  
\eeq
For $\mu = 0$ it follows that
\beq
  \psit_{0,\:\!a}(x)
= \int\dd u
= u
= \arcosh\frac{x}{a}
= \ln\frac{x+\sqrt{x^2-a^2}}{a}\,.  
\eeq
For $\mu = 1,3,5,\ldots$ from Eq.\ 2.412.4 in \textcite[Vol.\ 1, p.\ 127]{Gradstein&Ryshik-engl} we get
\begin{align} \label{Eq:135}
  \int\cosh^\mu u\:\dd u
= {} & \frac{\sinh u}{\mu}\left[\cosh^{\mu-1}u + \sum_{\nu=1}^{\frac{\mu-1}{2}}
  \left(\prod_{j=1}^\nu\frac{\mu+1-2j}{\mu-2j}\right)\cosh^{\mu-1-2\nu}u\right]\db\nonumber\\[0.05cm]
= {} & \frac{\cosh^{\mu-1}u\,\sinh u}{\mu}\left(1 + \sum_{\nu=1}^{\frac{\mu-1}{2}}\frac{1}{\cosh^{2\nu}u}\,
  \prod_{j=1}^\nu\frac{\mu+1-2j}{\mu-2j}\right),
\end{align}
whereas for $\mu = 2,4,6\ldots$ from Eq.\ 2.412.1 in \autocite[Vol.\ 1, p.\ 127]{Gradstein&Ryshik-engl} we get
\begin{align} \label{Eq:246}
  \int\cosh^\mu u\:\dd u
= {} & \frac{\cosh^{\mu-1}u\,\sinh u}{\mu}\left(1 + \sum_{\nu=1}^{\frac{\mu}{2}-1}\frac{1}{\cosh^{2\nu}u}\,
  \prod_{j=1}^\nu\frac{\mu+1-2j}{\mu-2j}\right) + \frac{(\mu-1)!!}{\mu!!}\,u\,.
\end{align}
The back substitution
\begin{gather*}
  \cosh u
= \frac{x}{a}\,,\quad
  \sinh u
= \frac{\sqrt{x^2-a^2}}{a}\,,\quad
  u
= \mathrm{arcosh}\,\frac{x}{a}
= \ln\left(\frac{x}{a} + \sqrt{\left(\frac{x}{a}\right)^2-1}\,\right)
\end{gather*}  
in \eqref{Eq:135} and \eqref{Eq:246} yields $\psit_{\mu,\:\!a}(x)$ for the remaining cases in Lemma \ref{Lem:psit}.
\end{proof}

\begin{examples} \label{Examples:psit}
From Lemma \ref{Lem:psit} one gets the following special cases
\begin{align*}
  \psit_{1,\:\!a}(x)
= {} & \sqrt{x^2-a^2}\,,\db\\[0.05cm]
  \psit_{2,\:\!a}(x)
= {} & \frac{1}{2}\left(x\,\sqrt{x^2-a^2} + a^2\ln\frac{x+\sqrt{x^2-a^2}}{a}\right),\db\\[0.05cm]
  \psit_{3,\:\!a}(x)
= {} & \frac{1}{3}\left(x^2+2a^2\right)\sqrt{x^2-a^2}\,,\db\\[0.05cm]
  \psit_{4,\:\!a}(x)
= {} & \frac{1}{8}\left[x\left(2x^2+3a^2\right)\sqrt{x^2-a^2} 
  + 3a^4\ln\frac{x+\sqrt{x^2-a^2}}{a}\right],\db\\[0.05cm]
  \psit_{5,\:\!a}(x)
= {} & \frac{1}{15}\left(3x^4+4a^2x^2+8a^4\right)\sqrt{x^2-a^2}\,,\db\\[0.05cm]
  \psit_{6,\:\!a}(x)
= {} & \frac{1}{48}\left[x\left(8x^4+10a^2x^2+15a^4\right)\sqrt{x^2-a^2} 
  + 15a^6\ln\frac{x+\sqrt{x^2-a^2}}{a}\right].       
\end{align*}
\end{examples}

\begin{remark}
The integral
\beq
  \int\frac{x^{\mu-1}}{\sqrt{1-(a/x)^2}}\,\dd x
= \int\frac{x^\mu}{\sqrt{x^2-a^2}}\,\dd x
\eeq
is also used and discussed in the context of computing the moments of a point distance associated with triangles in \textcite[p.\ 4]{Li&Qiu}.
\end{remark}

\begin{lem} \label{Lem:sigmat}
For $\mu = 0,1,2,\ldots$, an antiderivative $\sigmat_{\mu,\:\!a}(x)$ of the function
\beq
  \sigma_{\mu,\:\!a}
  \colon [a,\infty) \:\rightarrow\: \R\,,\quad
  x \:\mapsto\: \sigma_{\mu,\:\!a}(x) := x^\mu\arcsin\frac{a}{x}\,,\quad
  a \ge 0\,,  
\eeq
(from Theorem \ref{Thm:CLD}) is given by
\beq
  \sigmat_{\mu,\:\!a}(x)
= \left\{
  \begin{array}{ccl}
	\dfrac{1}{\mu+1}\left(x^{\mu+1}\arcsin\dfrac{a}{x} + a\,\psit_{\mu,\:\!a}(x)\right) &
	\mbox{if} & a > 0\,,\\[0.4cm] 
	0 & \mbox{if} & a = 0\,,
  \end{array}
  \right.
\eeq
with $\psit_{\mu,\:\!a}(x)$ according to Lemma \ref{Lem:psit}.
\end{lem}

\begin{proof}
Clearly, $\sigmat_{\mu,\:\!0}(x) \equiv 0$.
If $a > 0$, then integration by parts gives
\beq
  \sigmat_{\mu,\:\!a}(x)
= \int x^\mu\arcsin\frac{a}{x}\:\dd x
= \frac{1}{\mu+1}\left(x^{\mu+1}\arcsin\frac{a}{x} + a\,\psit_{\mu,\:\!a}(x)\right). \qedhere   
\eeq
\end{proof}

\begin{lem} \label{Lem:taut}
For $\mu = -1,0,1,2,\ldots$, an antiderivative $\taut_{\mu,\:\!a}(x)$ of the function
\beq
  \tau_{\mu,\:\!a}
  \colon [a,\infty) \:\rightarrow\: \R\,,\quad
  x \:\mapsto\: \tau_{\mu,\:\!a}(x) := x^\mu\,\sqrt{x^2-a^2}\,,\quad
  a \ge 0\,,  
\eeq
(from Theorem \ref{Thm:CLD}) is given by
\begin{align*}
  \taut_{\mu,\:\!a}(x)
= {} & \left\{
  \begin{array}{ccl}
	\ds{\frac{1}{\mu+2}\left(x^{\mu+1}\,\sqrt{x^2-a^2} - a^2\,\psit_{\mu,\:\!a}(x)\right)} &
	\mbox{if} & a > 0\,,\\[0.4cm] 
	\ds{\frac{x^{\mu+2}}{\mu+2}} & \mbox{if} & a = 0\,,
  \end{array}
  \right.	  
\end{align*}
with $\psit_{\mu,\:\!a}(x)$ according to Lemma \ref{Lem:psit}.
\end{lem}

\begin{proof}
We have
\beq
  \taut_{\mu,\:\!0}(x)
= \int x^{\mu+1}\,\dd x
= \frac{x^{\mu+2}}{\mu+2}\,.  
\eeq
In the case $a > 0$ we have
\beqn \label{Eq:tauta}
  \taut_{\mu,\:\!a}(x)
= \int x^\mu\,\sqrt{x^2-a^2}\,\dd x   
= \int\frac{x^\mu\,(x^2-a^2)}{\sqrt{x^2-a^2}}\,\dd x
= \int\frac{x^{\mu+2}}{\sqrt{x^2-a^2}}\,\dd x - a^2\,\psit_{\mu,\:\!a}(x)\,. 
\eeqn
Integration by parts with
\beq
\begin{array}{c@{\:=\:}l@{\qquad}c@{\;=\;}l}
  u  & x^{\mu+1}\,,    & v' & \ds{\frac{x}{\sqrt{x^2-a^2}}}\,,\\[0.35cm]
  u' & (\mu+1)x^\mu\,, & v  & \ds{\sqrt{x^2-a^2}}	
\end{array}  
\eeq
yields
\begin{align*}
  \int\frac{x^{\mu+2}}{\sqrt{x^2-a^2}}\,\dd x
= {} & x^{\mu+1}\,\sqrt{x^2-a^2} - (\mu+1)\int x^\mu\,\sqrt{x^2-a^2}\,\dd x\db\\[0.05cm]
= {} & x^{\mu+1}\,\sqrt{x^2-a^2} - (\mu+1)\,\taut_{\mu,\:\!a}(x)\,. 
\end{align*}
Putting this into \eqref{Eq:tauta} provides
\beq
  \taut_{\mu,\:\!a}(x)
= x^{\mu+1}\,\sqrt{x^2-a^2} - (\mu+1)\,\tau_{\mu,\:\!a}(x) - a^2\,\psit_{\mu,\:\!a}(x)\,,  
\eeq
hence
\beq
  \taut_{\mu,\:\!a}(x)
= \frac{1}{\mu+2}\left(x^{\mu+1}\,\sqrt{x^2-a^2} - a^2\,\psit_{\mu,\:\!a}(x)\right). \qedhere  
\eeq 
\end{proof}

\begin{remark}
I do not use $\taut_{\mu,\:\!a}(x)$ in the closer form
\beq
  \taut_{\mu,\:\!a}(x)
= \frac{1}{\mu+1}\left(x^{\mu+1}\,\sqrt{x^2-a^2} - \int\frac{x^{\mu+2}}{\sqrt{x^2-a^2}}\,\dd x\right)
= \frac{1}{\mu+1}\left(x^{\mu+1}\,\sqrt{x^2-a^2} - \psit_{\mu+2,\:\!a}(x)\right),  
\eeq
which follows immediately by partial integration from the first integral in \eqref{Eq:tauta}, because it does not hold for the case $\mu = -1$, which is also needed in the following.  
\end{remark}
% !TeX root = MDRP_0.tex

\section{The point distance density function}
\label{Sec:PDF}

\begin{thm} \label{Thm:PDF}
With $\LP$ and $\AP$ according to \eqref{Eq:L_and_A}, $\alpha$, $K$, $\ell_k$, $\la$, $h_k$, $p_k$, $q_k$, $s_k$ as in Theorem \ref{Thm:CLD},
\beq
  \sigmat_{1,\:\!a}(x)
= \frac{1}{2}\left(x^2\arcsin\frac{a}{x} + a\,\sqrt{x^2-a^2}\right),\qquad
  \taut_{-1,\:\!a}(x)
= \sqrt{x^2-a^2} + a\arcsin\frac{a}{x}\,,  
\eeq
the density function $\fD$ of the distance $\D_{n,\:\!r}$ between two uniformly and independently distributed random points in $\Po$, $n = 3,4,\ldots$, is given by
\beq
  \fD(x)
= \left\{
  \begin{array}{l@{\quad\mbox{if}\quad}l}
	\dfrac{2x}{\AP}\left[\pi-\dfrac{\LP\,\phi^\flat(x)}{\AP\,\ell_1}\right] &
		x \in [0,\ell_{K+1})\,,\\[0.5cm]
	0 & x \in \R\setminus[0,\ell_{K+1})\,,		
  \end{array}
  \right. 
\eeq
where
\beq
  \phi^\flat(x)
= \sum_{\nu=0}^{k-1} J_\nu^\flat(\ell_{\nu+1}) + J_k^\flat(x)
  \quad\mbox{if}\quad
  \ell_k \le x < \ell_{k+1}\,,\:\: k = 0,\ldots,K 
\eeq
with
\beq
  J_k^\flat(x)
= \left\{
  \begin{array}{lcl}
	H_{0,\:\!k}^\flat(x) & \mbox{if} & k=0\,\wedge\,x<\la\,,\\[0.15cm]
	H_{1,\:\!k}^\flat(x) - H_{1,\:\!k}^\flat(\ell_k) & \mbox{if} & (\mbox{$n$ is even}\,\wedge\,1 \le k \le K-1)\:\:
	\vee\\[0.05cm]
	& & (\mbox{$n$ is odd}\,\wedge\,1 \le k \le K\,\wedge\,x < \la)\,,\\[0.15cm]
	H_{2,\:\!k}^\flat(x) - H_{2,\:\!k}^\flat(\ell_k) & \mbox{if} & \mbox{$n$ is even}\,\wedge\,k = K\,,\\[0.15cm]
	H_{3,\:\!k}^\flat(x) - H_{3,\:\!k}^\flat(\la) + H_{0,\:\!k}^\flat(\la) & \mbox{if} & 
		n=3\,\wedge\,k = 0\,\wedge\,x\ge\la\,,\\[0.15cm]
	H_{3,\:\!k}^\flat(x) - H_{3,\:\!k}^\flat(\la) + H_{1,\:\!k}^\flat(\la) - H_{1,\:\!k}^\flat(\ell_k)& \mbox{if} & 
		\mbox{$n$ is odd}\,\wedge\,n>3\,\wedge\,x\ge\la\,,	
  \end{array}
  \right.   
\eeq
and
\begin{align*}
  H_{0,\:\!k}^\flat(x)
= {} & \ell_1\,x - \left[1+\alpha\left(\tan\alpha-\cot\alpha\right)\right]x^2/4\,,\db\\[0.15cm]  
  H_{1,\:\!k}^\flat(x)
= {} & p_k\,\sigmat_{1,\:\!h_k}(x)
  - \left[k\alpha\cot(2k\alpha)-(k+1)\alpha\cot(2(k+1)\alpha)\right]x^2/2\\[0.05cm]
& - (h_k\,p_k+2rq_k\cos\alpha)\,\taut_{-1,\:\!h_k}(x)\,,\db\\[0.15cm]
  H_{2,\:\!k}^\flat(x)
= {} & \left[1/2-k\alpha\cot(2k\alpha)\right]x^2/2
  + \cot(2k\alpha)\,\sigmat_{1,\:\!h_k}(x)
  - \left(h_k-2r\cos\alpha\right)h_k\ln x\\[0.05cm]
& - \left[h_k\cot(2k\alpha) + 2r\cos\alpha\tan(k\alpha)\right]\taut_{-1,\:\!h_k}(x)\,,\db\\[0.15cm]    
  H_{3,\:\!k}^\flat(x)
= {} & p_k\,\sigmat_{1,\:\!h_k}(x) + 2\cot(2(k+1)\alpha)\,\sigmat_{1,\:\!\la}(x)
  - \left[(\pi-\alpha)\cot(2(k+1)\alpha)+s_k\right]x^2/2\\[0.05cm]
& - (h_k\,p_k+2rq_k\cos\alpha)\,\taut_{-1,\:\!h_k}(x)\\[0.05cm]
& - 2\cos\alpha\left[2r\cos(\alpha/2)\sec((k+1)\alpha)-\la\csc(2(k+1)\alpha)\right]\taut_{-1,\:\!\la}(x)\,.
\end{align*}
\end{thm}

\begin{proof}
From Piefke's \autocite*[p.\ 130]{Piefke} result it immediately follows that the density functions $\fL$ and $\fD$ are connected by
\beq
  \fD(x)
= \frac{2\LP\,x}{\AP^2} \int_x^d (s-x) \fL(s)\, \dd s\,.  
\eeq
In \textcite{Baesel:Random_chords} (see pp.\ 10-11) we concluded that
\beqn \label{Eq:f_Delta}
  \fD(x)
= \frac{2x}{\AP} \left[\pi - \frac{\LP\,\phi^\natural(x)}{\AP}\right]  
\eeqn
where
\beqn \label{Eq:phi^natural}
  \phi^\natural(x)
:= x - \int_0^x \FL(s)\, \dd s
= \int_0^x \left(1-\FL(s)\right) \dd s\,.
\eeqn
So for $\ell_k \le x \le \ell_{k+1}$, $k = 0,\dots,K$, we have
\beqn \label{Eq:phi^b}
  \ell_1\,\phi^\natural(x)
= \phi^\flat(x)
:= \sum_{\nu=0}^{k-1}\, J_\nu^\flat(\ell_{\nu+1}) + J_k^\flat(x)  	
\eeqn
where
\beqn \label{Eq:J_nu^b}
  J_\nu^\flat(x)
= \int_{\ell_\nu}^x H_\nu(s)\, \dd s  
\eeqn
with $H_\nu$ according to Theorem \ref{Thm:CLD} (the sum in \eqref{Eq:phi^b} is empty if $k=0$).
Considering the necessary case distinctions, the result of Theorem \ref{Thm:PDF} follows from evaluating the integral \eqref{Eq:J_nu^b} for $\nu = 0,\ldots,K$. 
Note that the condition $k=0\,\wedge\,x<\la$ is an abbreviation for the equivalent, more detailed condition
\beq
  (n=3\,\wedge\,k=0\le x<\la)\,\vee\,(n>3\,\wedge\,k=0)\,.
\eeq
The expression for $\sigmat_{1,\:\!a}(x)$ follows from Lemma \ref{Lem:sigmat} with $\psit_{1,\:\!a}(x)$ in Examples \ref{Examples:psit}, whereas the expression for $\taut_{-1,\:\!a}(x)$ follows from Lemma \ref{Lem:taut} with $\psit_{-1,\:\!a}(x)$ according to Lemma \ref{Lem:psit}.  
\end{proof}

Examples for graphs of density functions are shown in Figures \ref{Abb:Diagramm_P_(7,r)} and \ref{Abb:Diagramm_P_(8,r)}. $f_{n,\:\!r}(x) = \dd F_{n,\:\!r}(x)/\dd x$ denotes the chord length density function.

\begin{SCfigure}[0.6][h]
  \includegraphics[width=0.55\textwidth]{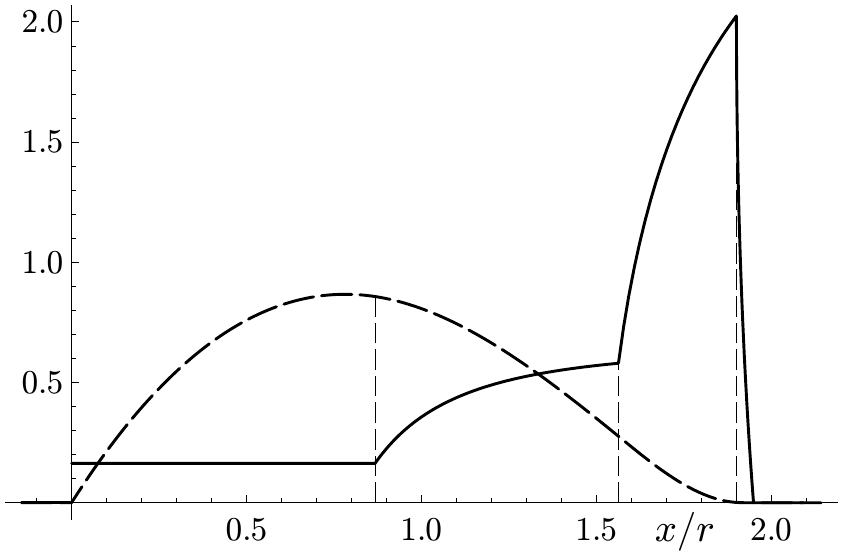}
  \caption{Graphs of density functions for $\mathcal{P}_{7,\:\!r}$: $r\times f_{7,\:\!r}(x)$, and $r\times g_{7,\:\!r}(x)$ (dashed)}
  \label{Abb:Diagramm_P_(7,r)}
\end{SCfigure}

\begin{SCfigure}[0.6][h]
  \includegraphics[width=0.55\textwidth]{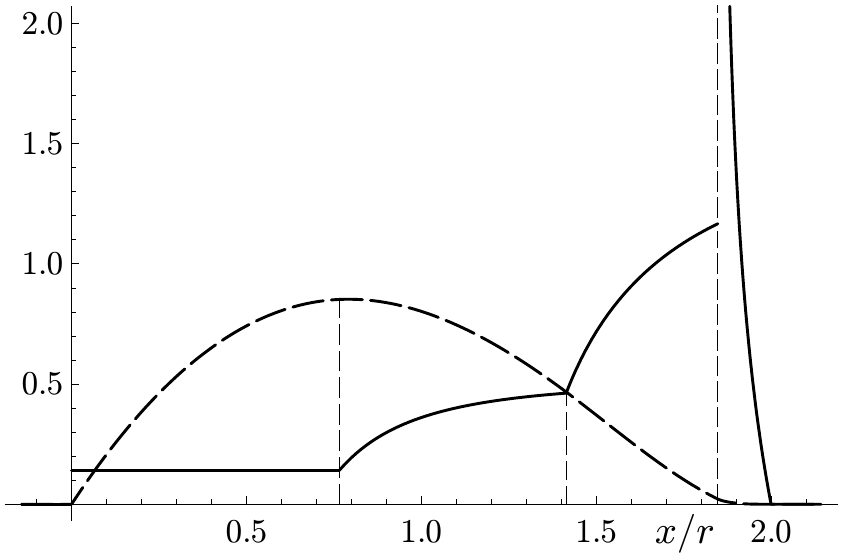}
  \caption{Graphs of density functions for $\mathcal{P}_{8,\:\!r}$: $r\times f_{8,\:\!r}(x)$, and $r\times g_{8,\:\!r}(x)$ (dashed)}
  \label{Abb:Diagramm_P_(8,r)}
\end{SCfigure}
% !TeX root = MDRP_0.tex

\newpage

\section{Moments of the distance}
\label{Sec:Moments_distance}

\begin{thm} \label{Thm:Moments}
With
\begin{itemize}[leftmargin=0.5cm]
\setlength{\itemsep}{-2pt}
\item $\LP$, $\AP$ according to \eqref{Eq:L_and_A},
\item $\alpha$, $K$, $\ell_k$, $\la$, $h_k$, $p_k$, $q_k$, $s_k$ as in Theorem \ref{Thm:CLD},
\item $\sigmat_{\mu,\:\!a}(x)$ according to Lemma \ref{Lem:sigmat},
\item $\taut_{\mu,\:\!a}(x)$ according to Lemma \ref{Lem:taut},
\end{itemize}
the moments of the distance between two uniformly and independently distributed random points in $\Po$ are given by
\beq
  M_m(\Po)
= \frac{2\LP}{(m+2)\,\AP^2\,\ell_1}\,\sum_{k=0}^K \widetilde{J}_{m,\:\!k}(\ell_{k+1})\,,\quad
  m = -1,0,1,2,\ldots  
\eeq
where
\beq
  \widetilde{J}_{m,\:\!k}(x)
= \left\{
  \begin{array}{l@{\;\;\,}c@{\;\;\,}l}
	\Ht_{m,\:\!k}^{(0)}(x) & \mbox{if} & k = 0 \,\wedge\, x < \la\,,\\[0.2cm]
	\Ht_{m,\:\!k}^{(1)}(x) - \Ht_{m,\:\!k}^{(1)}(\ell_k) &
	\mbox{if} & \mbox{$(n$ is even $\,\wedge\,$ $1 \le k \le K-1)$ $\,\vee$}\\
	&  & \mbox{$(n$ is odd $\,\wedge\,$ $1 \le k \le K$ $\,\wedge\,$ $x<\la)$}\,,\\[0.2cm]
	\Ht_{m,\:\!k}^{(2)}(x) - \Ht_{m,\:\!k}^{(2)}(\ell_k) &
	\mbox{if} & \mbox{$n$ is even $\,\wedge\,$ $k = K$}\,,\\[0.2cm]  
    \Ht_{m,\:\!k}^{(3)}(x) - \Ht_{m,\:\!k}^{(3)}(\la) + \Ht_{m,\:\!k}^{(0)}(\la) &
    \mbox{if} & \mbox{$n = 3$ $\,\wedge\,$ $k = 0$
    $\,\wedge\,$ $x \ge \la$}\,,\\[0.2cm]
    \Ht_{m,\:\!k}^{(3)}(x) - \Ht_{m,\:\!k}^{(3)}(\la) + \Ht_{m,\:\!k}^{(1)}(\la) - \Ht_{m,\:\!k}^{(1)}(\ell_k) & 
    \mbox{if} & \mbox{$n$ is odd $\,\wedge\,$ $n > 3$ $\,\wedge\,$ $x \ge \la$}
  \end{array}
  \right.
\eeq
with
\begin{align*}
  \Ht_{m,\:\!k}^{(0)}(x)
= {} & \frac{\ell_1\,x^{m+3}}{m+3}
  - \frac{\left[1+\alpha\left(\tan\alpha-\cot\alpha\right)\right]x^{m+4}}{2(m+4)}\,,\db\\[0.15cm]
  \Ht_{m,\:\!k}^{(1)}(x)
= {} & p_k\,\sigmat_{m+3,\:\!h_k}(x) - \left[k\alpha\cot(2k\alpha)
  - (k+1)\alpha\cot(2(k+1)\alpha)\right]\frac{x^{m+4}}{m+4}\\[0.05cm]
& - \left(h_k\,p_k+2rq_k\cos\alpha\right)\taut_{m+1,\:\!h_k}(x)\,,\db\\[0.15cm]
  \Ht_{m,\:\!k}^{(2)}(x)
= {} & \left(\frac{1}{2}-k\alpha\cot(2k\alpha)\right)\frac{x^{m+4}}{m+4}
  + \cot(2k\alpha)\,\sigmat_{m+3,\:\!h_k}(x) - \left(h_k-2r\cos\alpha\right)h_k\,\frac{x^{m+2}}{m+2}\\[0.05cm]
& - \left[h_k\cot(2k\alpha)+2r\cos\alpha\tan(k\alpha)\right]\taut_{m+1,\:\!h_k}(x)\,,\db\\[0.15cm]
  \Ht_{m,\:\!k}^{(3)}(x)
= {} & p_k\,\sigmat_{m+3,\:\!h_k}(x) + 2\cot(2(k+1)\alpha)\,\sigmat_{m+3,\:\!\la}(x)\\[0.05cm]
& - \left[(\pi-\alpha)\cot(2(k+1)\alpha)+s_k\right]\frac{x^{m+4}}{m+4}
  - \left(h_k\,p_k+2rq_k\cos\alpha\right)\taut_{m+1,\:\!h_k}(x)\\[0.15cm]
& - 2\cos\alpha\left[2r\cos(\alpha/2)\sec((k+1)\alpha)-\la\csc(2(k+1)\alpha)\right]\taut_{m+1,\:\!\la}(x)\,.
\end{align*}
\end{thm}

\begin{proof}
From
\beq
  M_m(\Po) = \E[\D_{n,\:\!r}^m]
= \int_0^d x^m\,\fD(x)\: \dd x  
\eeq
with \eqref{Eq:f_Delta} and \eqref{Eq:phi^natural} we get
\begin{align*}
  M_m(\Po)
= {} & \frac{2\pi}{\AP}\int_0^d x^{m+1}\,\dd x
  - \frac{2\LP}{\AP^2}\int_{x=0}^d x^{m+1}\int_{s=0}^x(1-\FL(s))\,\dd s\,\dd x\db\\[0.05cm] 
= {} & \frac{2\pi d^{m+2}}{(m+2)\AP}
  - \frac{2\LP}{\AP^2}\int_{x=0}^d x^{m+1}\int_{s=0}^x(1-\FL(s))\,\dd s\,\dd x\,.  
\end{align*}
Integration by parts with
\beq
\begin{array}{c@{\:=\:}l@{\qquad}c@{\:=\:}l}
  u  & \int_0^x(1-F(s))\,\dd s\,, & v' & x^{m+1}\,,\\[0.25cm]
  u' & 1 - F(x)\,,				  & v  & x^{m+2}/(m+2)                 
\end{array} 
\eeq
yields
\begin{align*}
  M_m(\Po)
= {} & \frac{2\pi d^{m+2}}{(m+2)\AP}
  - \frac{2\LP}{\AP^2} \bigg[
  \left(\frac{x^{m+2}}{m+2}\,\int_0^x(1-\FL(s))\,\dd s\right)\bigg|_{x=0}^d\\[0.05cm]
& - \frac{1}{m+2}\int_0^d x^{m+2}\,(1-\FL(x))\,\dd x\bigg]\db\\[0.05cm]  
= {} & \frac{2\pi d^{m+2}}{(m+2)\AP}
  - \frac{2\LP\,d^{m+2}}{(m+2)\AP^2}\int_0^d(1-\FL(s))\,\dd s\\[0.05cm]
& + \frac{2\LP}{(m+2)\AP^2}\int_0^d x^{m+2}\,(1-\FL(x))\,\dd x\,.
\end{align*}
Now, integration by parts with
\beq
\begin{array}{c@{\:=\:}l@{\qquad}c@{\:=\:}l}
  u  & 1 - \FL(s)\,, & v' & 1\,,\\[0.25cm]
  u' & -\fL(s)\,,	 & v  & s                 
\end{array}
\eeq
gives
\beq
  \int_0^d (1-\FL(s))\,\dd s
= \underbrace{s\,(1-\FL(s))\,\Big|_0^d}_{\ds 0} \;\,+\, 
  \underbrace{\int_0^d s\,\fL(s)\,\dd s}_{\ds\E[\La_{n,\:\!r}]}\,. 
\eeq
With 
\beq
  \E[\La_{n,\:\!r}]
= \frac{\pi\AP}{\LP}\qquad
  \mbox{\autocite[p.\ 30]{Santalo}}  
\eeq
we have found
\beqn \label{Eq:M_m}
  M_m(\Po)
= \frac{2\LP}{(m+2)\AP^2}\int_0^d x^{m+2}\,(1-\FL(x))\,\dd x\,.
\eeqn
It follows that
\begin{align*}
  M_m(\Po)
= {} & \frac{2\LP}{(m+2)\AP^2}\,\sum_{k=0}^K\,\int_{\ell_k}^{\ell_{k+1}} x^{m+2}\,(1-\FL(x))\,\dd x\db\\[0.05cm]
= {} & \frac{2\LP}{(m+2)\AP^2\,\ell_1}\,\sum_{k=0}^K\,
  \underbrace{\int_{\ell_k}^{\ell_{k+1}} x^{m+2}\,H_k(x)\, \dd x}_{\ds =:\widetilde{J}_{m,\:\!k}(\ell_{k+1})}.
\end{align*}
with $H_k(x)$, $k \in \{0,1,\ldots,N\}$, from Theorem \ref{Thm:CLD}.
So it remains to determine the antiderivatives
\beq
  \Ht_{m,\:\!k}^{(q)}(x)
:= \int x^{m+2}\,H_k(x)\,\dd x\,,  
\eeq
where the $q$ is used to indicate the necessary case distinctions.

The antiderivatives $\sigmat_{m+3,\:\!a}(x)$ and $\taut_{m+1,\:\!a}(x)$ are used because of
\beq
  \int x^{m+2}\,\sigma_{1,\:\!a}(x)\,\dd x
= \int x^{m+2}\,x\arcsin\frac{a}{x}\,\dd x   
= \sigmat_{m+3,\:\!a}(x)
\eeq
and
\beq
  \int x^{m+2}\,\tau_{-1,\:\!a}(x)\,\dd x
= \int x^{m+2}\,\frac{\sqrt{x^2-a^2}}{x}\,\dd x   
= \taut_{m+1,\:\!a}(x)\,,
\eeq
respectively.
\end{proof}

\begin{remark}
Note that, although the expressions are in general quite long, Theorem \ref{Thm:Moments} always provides closed form expressions for the moments of every $\Po$.
\end{remark}

\begin{remark}
The relation \eqref{Eq:M_m} also follows from
\beqn \label{Eq:S_m-T_m-a}
  T_m = \frac{2}{(m+2)(m+3)}\, S_{m+3}\,,\quad
  m = -1,0,1,2,\,\ldots,
\eeqn
where $S_m$ and $T_m$ are the $m$-th chord power integral and the $m$-th distance power integral, respectively (\textcite[pp.\ 19-20]{Blaschke:Integralgeometrie}, \textcite[pp.\ 46-47]{Santalo}, see also \textcite[p.\ 364, proof of Thm.\ 8.6.6{;} p.\ 374]{Schneider&Weil2008}).
With
\beqn \label{Eq:S_m-T_m-b}
  \E[\La_{n,\:\!r}^m] = \frac{S_m(\Po)}{S_0(\Po)} = \frac{S_m(\Po)}{\LP}
  \quad\mbox{and}\quad
  M_m(\Po) = \E[\D_{n,\:\!r}^m] = \frac{T_m(\Po)}{T_0(\Po)} = \frac{T_m(\Po)}{\AP^2}
\eeqn
one gets
\begin{align*}
  M_m(\Po)
= {} & \frac{2}{(m+2)(m+3)}\,\frac{\LP}{\AP^2}\int_0^d x^{m+3}\, \dd\FL(x)\nonumber\db\\[0.05cm]
= {} & \frac{2\LP}{(m+2)\AP^2}\int_0^d x^{m+2}\,(1-\FL(x))\,\dd x  
\end{align*}
(see also \textcite[Eq.\ (2.3)]{Aharonyan&Ohanyan}).

\textcite{Heinrich2009} derived a general formula for the second-order chord power integrals, $S_2$, for $\Po$.
This gives $T_{-1}$ and hence $\E[\Delta_{n,\:\!r}^{-1}]$ (see also Section \ref{Sec:CPI}).
\end{remark}

%%%%%%%%%%%%%%%%%%%%%%%%%%%%%%%%%%%%%
%%%%%%%%%% Mean distances %%%%%%%%%%%
%%%%%%%%%%%%%%%%%%%%%%%%%%%%%%%%%%%%% 

\begin{cor} \label{Cor:MD}
The mean distance $M_1(\Po)$ between two uniformly and independently distributed random points in $\Po$ is given by Theorem \ref{Thm:Moments} with $m=1$.
If $a > 0$, then
\begin{align*}
  \sigmat_{4,\:\!a}(x)
= {} & \frac{1}{40}\left[8\,x^5\arcsin\frac{a}{x}
  + a\,x\left(2x^2+3a^2\right)\sqrt{x^2-a^2}
  + 3\,a^5\ln\frac{x+\sqrt{x^2-a^2}}{a}\right],\db\\[0.15cm]
  \taut_{2,\:\!a}(x)
= {} & \frac{1}{8}\left[x\left(2x^2-a^2\right)\sqrt{x^2-a^2}
  - a^4\ln\frac{x+\sqrt{x^2-a^2}}{a}\right].
\end{align*}
\end{cor}
% !TeX root = MDRP_0.tex

\section{Examples of closed form expressions for the moments}
\label{Sec:Closed_form_expressions}

For abbreviation we denote the moments in the following with $M_m$ instead of $M_m(\Po)$.
From the context it will be clear for which polygon $\Po$ the moments $M_m$ are valid.

%33333333333333333333333333333333
\textbullet\, 3-gon (equilateral triangle)
\begin{gather*}
  \frac{M_{-1}}{r^{-1}}
= \frac{4\sqrt{3}\ln 3}{3}\,,\quad  
  \frac{M_0}{r^0}
= 1\,,\quad
  \frac{M_1}{r}
= \frac{\sqrt{3}}{5} + \frac{3\sqrt{3}\ln 3}{20}\,,\quad  
  \frac{M_2}{r^2}
= \frac{1}{2}\,,\db\\[0.15cm]
  \frac{M_3}{r^3}
= \frac{51\sqrt{3}}{280} + \frac{81\sqrt{3}\ln 3}{1120}\,,\quad
  \frac{M_4}{r^4}
= \frac{9}{20}\,,\quad
  \frac{M_5}{r^5}
= \frac{383\sqrt{3}}{1792} + \frac{405\sqrt{3}\ln 3}{7168}\,,\quad
  \frac{M_6}{r^6}
= \frac{747}{1400}\,,\db\\[0.15cm]
  \frac{M_7}{r^7}
= \frac{6669\sqrt{3}}{22528} + \frac{5103\sqrt{3}\ln 3}{90112}\,,\quad
  \frac{M_8}{r^8}
= \frac{261}{350}\,,\quad\db\\[0.15cm]
  \frac{M_9}{r^9}
= \frac{977913\sqrt{3}}{2129920} + \frac{1240029\sqrt{3}\ln 3}{18743296}\,,\quad
  \frac{M_{10}}{r^{10}}
= \frac{2511}{2156}
\end{gather*}
With side length $\ell_1 = \sqrt{3}\,r$ we have
\beq
  \frac{M_1}{\ell_1}
= \frac{1}{5} + \frac{3\ln 3}{20}
= \frac{3}{5} \left(\frac{1}{3} + \frac{1}{4}\,\ln 3\right)
\approx 0.36479184330021645371\,.  
\eeq
This result is due to \textcite[p.\ 206]{Czuber1884} (see also \textcite[p.\ 49]{Santalo}, \textcite{Weisstein:Triangle_Line_Picking}, \textcite{OEIS:A093064}, \textcite{Zhuang&Pan12}, and \textcite{Aharonyan&Ohanyan}).\\[0.2cm]
%44444444444444444444444444444444
\textbullet\, 4-gon (square)
\begin{gather*}
  \frac{M_{-1}}{r^{-1}}
= -\frac{4}{3} + \frac{2\sqrt{2}}{3} + 2\sqrt{2}\ln\left(1+\sqrt{2}\right)\,,\quad
  \frac{M_0}{r^0}
= 1\,,\quad  
  \frac{M_1}{r}
= \frac{2}{15} + \frac{2\sqrt{2}}{15} + \frac{\sqrt{2}\ln\left(1+\sqrt{2}\right)}{3}\,,\db\\[0.05cm]
  \frac{M_2}{r^2}
= \frac{2}{3}\,,\quad
  \frac{M_3}{r^3}
= \frac{34}{105} + \frac{8\sqrt{2}}{105} + \frac{\sqrt{2}\ln\left(1+\sqrt{2}\right)}{5}\,,\quad
  \frac{M_4}{r^4}
= \frac{34}{45}\,,\\[0.15cm]
  \frac{M_5}{r^5}
= \frac{73}{126} + \frac{4\sqrt{2}}{63} + \frac{5\sqrt{2}\ln\left(1+\sqrt{2}\right)}{28}\,,\quad
  \frac{M_6}{r^6}
= \frac{116}{105}\,,\\[0.15cm]
  \frac{M_7}{r^7}
= \frac{3239}{2970} + \frac{32\sqrt{2}}{495} + \frac{7\sqrt{2}\ln\left(1+\sqrt{2}\right)}{36}\,,\quad
  \frac{M_8}{r^8}
= \frac{2992}{1575}\,,\\[0.15cm]
  \frac{M_9}{r^9}
= \frac{1721}{780} + \frac{32\sqrt{2}}{429} + \frac{21\sqrt{2}\ln\left(1+\sqrt{2}\right)}{88}\,,\quad
  \frac{M_{10}}{r^{10}}
= \frac{7648}{2079}
\end{gather*}
With side length $\ell_1 = \sqrt{2}\,r$ we have
\beq
  \frac{M_1}{\ell_1}
= \frac{1}{15}\left[\sqrt{2} + 2 + 5\,\ln\left(1+\sqrt{2}\right)\right]
\approx 0.52140543316472067833\,.\\[0.2cm]
\eeq
This result is due to \textcite[pp.\ 202-204, Problem XV]{Czuber1884}.
It can also be found in \textcite{Ghosh}, \textcite[p.\ 49]{Santalo}, \textcite[p.\ 479]{Finch2003}, \textcite{Weisstein:Square_Line_Picking}, \textcite{OEIS:A091505}, and \textcite{Aharonyan&Ohanyan}.\\[0.2cm]  
%55555555555555555555555555555555 
\textbullet\, 5-gon (pentagon)
\begin{gather*}
  \frac{M_1}{r}
= \frac{\sqrt{2}}{480}\,\sqrt{5+\sqrt{5}} \left[24\sqrt{5} - 8 - \left(35+9\sqrt{5}\right)\ln 5
  - 2\left(25+11\sqrt{5}\right)\ln\left(\sqrt{5}-2\right)\right],\db\\[0.15cm]
  \frac{M_2}{r^2}
= \frac{7}{12} + \frac{\sqrt{5}}{12}\,,\quad
  \frac{M_4}{r^4}
= \frac{47}{72} + \frac{11\sqrt{5}}{72}\,,\quad
  \frac{M_6}{r^6}
= \frac{167}{168} + \frac{2\sqrt{5}}{7}\,,\db\\[0.15cm]
  \frac{M_8}{r^8}
= \frac{65}{36} + \frac{145\sqrt{5}}{252}\,,\quad
  \frac{M_{10}}{r^{10}}
= \frac{427375}{116424} + \frac{625\sqrt{5}}{504}
\end{gather*}
%66666666666666666666666666666666
\textbullet\, 6-gon (hexagon)
\begin{gather*}
  \frac{M_{-1}}{r^{-1}}
= -\frac{16}{9} + \frac{8}{3\sqrt{3}} + \frac{22\ln 3}{9}- \frac{4\ln\left(2+\sqrt{3}\right)}{9}\,,\quad
  \frac{M_0}{r^0} = 1\,,\db\\[0.05cm]  
  \frac{M_1}{r}
= -\frac{7}{90} + \frac{7}{10\sqrt{3}} + \frac{19\ln 3}{40} - \frac{\ln\left(2+\sqrt{3}\right)}{60}\,,\quad 
  \frac{M_2}{r^2}
= \frac{5}{6}\,,\db\\[0.15cm]
  \frac{M_3}{r^3}
= \frac{817}{5040} + \frac{117\sqrt{3}}{560} + \frac{867\ln 3}{2240}
  - \frac{3\ln\left(2+\sqrt{3}\right)}{1120}\,,\quad
  \frac{M_4}{r^4}
= \frac{209}{180}\,,\db\\[0.15cm]
 \frac{M_5}{r^5}
= \frac{146431}{290304} + \frac{963\sqrt{3}}{3584} + \frac{7045\ln 3}{14336}
  - \frac{5\ln\left(2+\sqrt{3}\right)}{7168}\,,\quad
  \frac{M_6}{r^6}
= \frac{573}{280}\,,\db\\[0.15cm]
 \frac{M_7}{r^7}
= \frac{7886969}{6082560} + \frac{8541\sqrt{3}}{20480} + \frac{139797\ln 3}{180224}
  - \frac{21\ln\left(2+\sqrt{3}\right)}{90112}\,,\quad
  \frac{M_8}{r^8}
= \frac{13037}{3150}\,,\db\\[0.15cm]
  \frac{M_9}{r^9}
= \frac{1395486403}{421724160} + \frac{34152111\sqrt{3}}{46858240} + \frac{52256421\ln 3}{37486592}
  - \frac{1701\ln\left(2+\sqrt{3}\right)}{18743296}\,,\quad
  \frac{M_{10}}{r^{10}}
= \frac{76273}{8316}
\end{gather*}
The mean distance was obtained by \textcite{Zhuang&Pan11} in the form
\beq
  \frac{M_1}{\ell_1}
= \frac{7\sqrt{3}}{30} - \frac{7}{90} + \frac{1}{60}\left[28\ln\left(2\sqrt{3}+3\right) 
  + 29\ln\left(2\sqrt{3}-3\right)\right]   
\eeq
with side length $\ell_1 = r$. (Only the numerical approximation there is not correct.)
\textcite{Zhuang&Pan11} also derived the second moment.\\[0.2cm]
%88888888888888888888888888888888
\textbullet\, 8-gon (octagon)
\begin{gather*}
\begin{aligned}
  \frac{M_1}{r}
= {} & \frac{4\sqrt{2}\,\pi}{3} + 2\sqrt{2-\sqrt{2}}\left[\frac{1}{5}
  + \frac{13\sqrt{2}}{120} - \left(\frac{1}{20}+\frac{\sqrt{2}}{60}+\frac{2\pi}{3}\right)\sqrt{2+\sqrt{2}}
  \right.\\[0.05cm]
& -\frac{1}{48}\left(14+\frac{91\sqrt{2}}{10}\right)\ln\left(2+\sqrt{2}\right)
  + \frac{1}{48}\left(1+\frac{7\sqrt{2}}{10}\right)\ln\left(2-\sqrt{2}\right)\\[0.05cm]
& \left.+\, \frac{1}{20}\left(\frac{1}{3}-\frac{\sqrt{2}}{4}\right)\ln\left(2+\sqrt{2+\sqrt{2}}\,\right)
  + \frac{1}{40}\left(21+\frac{29\sqrt{2}}{2}\right)\ln\left(2+\sqrt{2-\sqrt{2}}\,\right)\right],
\end{aligned}\db\\[0.15cm]
  \frac{M_2}{r^2}
= \frac{2}{3} + \frac{\sqrt{2}}{6}\,,\quad  
  \frac{M_4}{r^4}
= \frac{77}{90} + \frac{16\sqrt{2}}{45}\,,\quad
 \frac{M_6}{r^6}
= \frac{157}{105} + \frac{23\sqrt{2}}{30}\,,\db\\[0.15cm]
  \frac{M_8}{r^8}
= \frac{326}{105} + \frac{928\sqrt{2}}{525}\,,\quad  
  \frac{M_{10}}{r^{10}}
= \frac{2132}{297} + \frac{3002\sqrt{2}}{693}
\end{gather*}
%10-10-10-10-10-10-10-10-10-10-10
\textbullet\, 10-gon (decagon)
\begin{gather*}
\begin{aligned}
  \frac{M_1}{r}
= {} & \frac{1}{600} \left[4\left(\sqrt{5125+2110\sqrt{5}}-24-11\sqrt{5}\right)
  - \left(505+239\sqrt{5}\right)\ln 2\right.\\[0.05cm]
& - 30\left(4+3\sqrt{5}\right)\ln 5
  - \left(205+107\sqrt{5}\right)\ln\left(1+\sqrt{5}\right)\\[0.05cm] 
& + \left(705+343\sqrt{5}\right)\ln\left(3+\sqrt{5}\right)
  - \left(105+47\sqrt{5}\right)\ln\left(\sqrt{5}+\sqrt{10+2\sqrt{5}}\,\right)\\[0.05cm]
& - \left(65-29\sqrt{5}\right)\ln\left(\sqrt{5}+\sqrt{10-2\sqrt{5}}\,\right)\\[0.05cm]
&  \left.+ \left(5+3\sqrt{5}\right)\ln\left(5-\sqrt{5}+2\sqrt{10-2\sqrt{5}}\,\right)
  \right],   
\end{aligned}\db\\[0.15cm]
  \frac{M_2}{r^2}
= \frac{3}{4} + \frac{\sqrt{5}}{12}\,,\quad  
  \frac{M_4}{r^4}
= \frac{121}{120} + \frac{73\sqrt{5}}{360}\,,\quad
  \frac{M_6}{r^6}
= \frac{1513}{840} + \frac{101\sqrt{5}}{210}\,,\db\\[0.05cm]
  \frac{M_8}{r^8}
= \frac{23983}{6300} + \frac{1073\sqrt{5}}{900}\,,\quad
  \frac{M_{10}}{r^{10}}
= \frac{149279}{16632} + \frac{223\sqrt{5}}{72}
\end{gather*} 
%12-12-12-12-12-12-12-12-12-12-10
\textbullet\, 12-gon (dodecagon)
\begin{gather*}
\begin{aligned}
  \frac{M_1}{r}
= {} & \frac{1}{1080\sqrt{2}} \left[504 - 660\sqrt{2} + 396\sqrt{3} - 4\sqrt{6}
  + 2\left(3+\sqrt{3}\right)\sqrt{2+\sqrt{3}}\right.\\[0.05cm]
& + 6\left(1+47\sqrt{3}\right)\sqrt{2-\sqrt{3}}
  - \frac{27}{2}\left(15+11\sqrt{3}\right)\ln 3
  - 4\left(27-\sqrt{3}\right)\ln\left(1+\sqrt{2}\right)\\[0.05cm]
& - 4\sqrt{3}\ln\left(2+\sqrt{3}\right) + 2\left(666+397\sqrt{3}\right)\ln\left(2-\sqrt{3}\right)\\[0.05cm]
& \left.- \left(33-19\sqrt{3}\right)\ln\left(2+\sqrt{2+\sqrt{3}}\,\right)
  + 3\left(899+523\sqrt{3}\right)\ln\left(2+\sqrt{2-\sqrt{3}}\,\right)
  \right],
\end{aligned}\db\\[0.15cm]
  \frac{M_2}{r^2}
= \frac{2}{3} + \frac{1}{2\sqrt{3}}\,,\quad
  \frac{M_4}{r^4}
= \frac{163}{180} + \frac{16}{15\sqrt{3}}\,,\quad
  \frac{M_6}{r^6}
= \frac{71}{42} + \frac{661}{280\sqrt{3}}\,,\db\\[0.05cm]
  \frac{M_8}{r^8}
= \frac{2357}{630} + \frac{424}{75\sqrt{3}}\,,\quad
  \frac{M_{10}}{r^{10}}
= \frac{19099}{2079} + \frac{521}{36\sqrt{3}}
\end{gather*}
% !TeX root = MDRP_0.tex

\section{Numerical values}
\label{Sec:Numerical_values}

Numerical values for mean distances are given in Table \ref{Ta:Moments}; as an example, numerical values of moments for $\mathcal{P}_{5,\:\!r}$ may be found in Table \ref{Ta:Moments-P_5}.
Using {\em Mathematica}, all of these values are obtained in two different ways:
\begin{itemize}[leftmargin=0.6cm] \mynobreakpar
\setlength{\itemsep}{-2pt}
\item[1)] from the analytical solution in Theorem \ref{Thm:Moments} together with Corollary \ref{Cor:MD},
\item[2)] with numerical integration of
\beq
\mbox{$
  M_m(\Po)
= \int_0^d x^m \fD(x)\,\dd x\,,\quad m = -1,0,1,2,\ldots, 
$} 
\eeq
with the density function $\fD$ from Theorem \ref{Thm:PDF}.
\end{itemize}
The integration with \texttt{NIntegrate} required the use of \texttt{WorkingPrecision} and \texttt{PrecisionGoal} with suitable parameters and the splitting of the interval $[0,d]$ into subintervals corresponding to the different cases. 

Fig.\ \ref{Fig:Diagram_Expected_distance} shows how the expected distances converge against the expected distance for the circle $\Cr$ with radius $r$ (see \eqref{Eq:M_1(C_r)-M_2(C_r)}, \eqref{Eq:M_1(C_r)}). 

\begin{table}[h]
\caption{Numerical values for the mean distance}
\label{Ta:Moments}
\centering
\begin{tabular}{|c|c|} \hline
\rule{0pt}{12pt}
$n$ & $M_1(\mathcal{P}_{n,\;\!r})/r$\\[2pt] \hline\hline
\phantom{1}3 & 0.631838006782679248439363765946266548228352566630216491351433566313728372374\\
\phantom{1}4 & 0.737378635076566348769571883395005078909753124947971245022700482462877666826\\
\phantom{1}5 & 0.793698195033753381760971632749605393899757772205869281310580089360628333136\\
\phantom{1}6 & 0.826258949490232082314283750323326010149318430219325083679749121623475379023\\
\phantom{1}7 & 0.846561326216093164027700615562504932765674717512492654162687095709688638441\\
\phantom{1}8 & 0.860007978015497247289475476698227488793382581154266797596077525063820243821\\
\phantom{1}9 & 0.869349677996368661272424554005308059692018424046355884971707394951786624055\\
10 & 0.876093016045821455437851478188806272714728045761204545124286721555684616290\\
11 & 0.881115231029789401140263676940968133233965295741220353166084636303706353850\\
12 & 0.884953782114214064113523711258487868874291600958495392058472857817523925124\\
13 & 0.887952286709934145445459685691276130847157335822373829281308109648830663248\\
14 & 0.890338490756764575735944950351064496029659712504470764018294183941844198540\\
15 & 0.892268061893522241535511225799568091834470923360328845899768230281232859177\\
16 & 0.893850267566140981445305963668775169957199844033103927893285119318229687792\\
17 & 0.895163602871345666236825912146252831958826216118785135695365798856133878291\\
18 & 0.896265616099618345601553161509629568460084862108878882774989021025191298962\\
19 & 0.897199265437573395493137709163413742006864641423080106927798068769440815162\\
20 & 0.897997136941287061624499349289436760861985366382141463318416018681802183354\\
21 & 0.898684307244840900050480580043224097039202683773713372510772133748302247016\\
22 & 0.899280326175269024941354083121964386591039853255838396881868160285060306349\\
23 & 0.899800615121290023756007063614270208943463748969153209258516013892113626863\\
24 & 0.900257469758660200325240434610054253924497221077824646172921619542988501865\\
25 & 0.900660789967557175495408085191424058901073528518182128993867898837236409609\\
26 & 0.901018618516579218784735836535111473157040513720977689621453066648252571482\\
27 & 0.901337543658947219455759999490321561595657334895730647458502404432253617671\\
28 & 0.901623003533134180530119238163701029410880699933247882599362042456633796468\\
29 & 0.901879518798576670046389739873503980147299817776988523533576670581933050086\\
30 & 0.902110872199677767358604753115452862213333195981189704602574660701374264234\\
$\infty$ & 0.905414787367226799040760964963637259573814873545707797320063112868390670986\\ \hline
\end{tabular}
\end{table}

\begin{table}[h]
\caption{Numerical moments for $\mathcal{P}_{5,\:\!r}$}
\label{Ta:Moments-P_5}
\centering
\begin{tabular}{|c|l|} \hline
\rule{0pt}{12pt}
$m$ & \multicolumn{1}{c|}{$M_m(\mathcal{P}_{5,\;\!r})/r^m$}\\[2pt] \hline\hline
$-1$ & 1.941532747349740237286829163767369291397446081205725687375578118779045940892\\
\phantom{1}0 & 1.000000000000000000000000000000000000000000000000000000000000000000000000000\\
\phantom{1}1 & 0.793698195033753381760971632749605393899757772205869281310580089360628333136\\
\phantom{1}2 & 0.769672331458315808034097805727606352953384863300960477022574770450876743803\\
\phantom{1}3 & 0.840599732769508183798403754882490669375942298135780251528211484773754115102\\
\phantom{1}4 & 0.994399274340245648062512643833944980414538916051760874541387079159940696972\\
\phantom{1}5 & 1.246445987635706663907926935585552909367630737102937503163072379765170177658\\
\phantom{1}6 & 1.632924184047558960878811524399412257744938578936626397410732546307767883516\\
\phantom{1}7 & 2.215335008390427051007858142854877211061180001285436481704877646702725059467\\
\phantom{1}8 & 3.092181971180434547537024531611250214836863738665362023886032145176688627847\\
\phantom{1}9 & 4.419361573049694692989749824140763338295447870581682927297936340850651262824\\
10 & 6.443751410804749342223892396954906260665664087034580725980990086289287076882\\ \hline
\end{tabular}
\end{table}

\clearpage

\begin{SCfigure}[0.36][h]
  \includegraphics[width=0.6\textwidth]{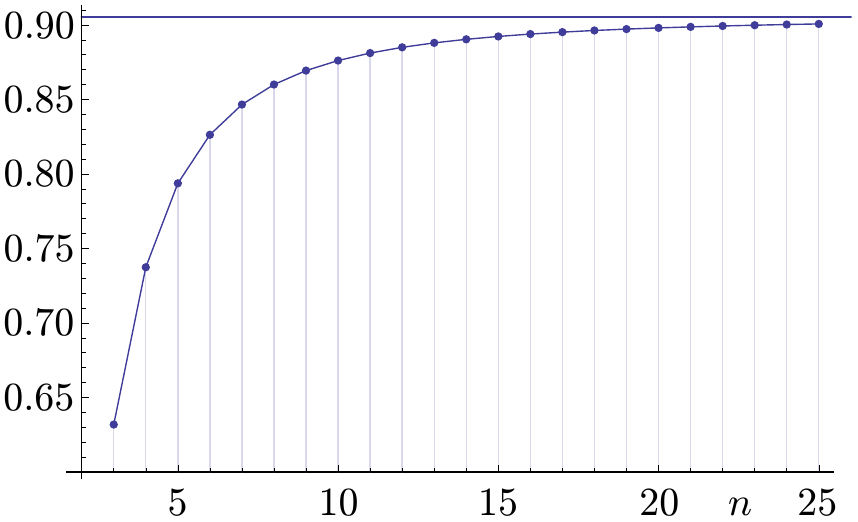}
  \caption{$M_1(\Po)/r$, $n=3,4,\ldots,25$, and $M_1(\Cr)/r$}
  \label{Fig:Diagram_Expected_distance}
\end{SCfigure}
% !TeX root = MDRP_0.tex

\section{Second moment and variance}
\label{Sec:Second_moment}

\begin{thm} \label{Thm:M_2(P_(n,r))}
The second moment of the point distance in $\Po$, $r = 3,4,\ldots$, is given by
\beq
  M_2(\Po)
= \frac{r^2}{3} \left(2 + \cos\frac{2\pi}{n}\right).
\eeq
\end{thm}

\begin{proof}
Due to \textcite[p.\ 217, Problem XXVII]{Czuber1884} we have
\beqn \label{Eq:Problem_XXVII}
  M_2(\Po)
= \frac{2I_p(\Po)}{A(\Po)}  
\eeqn
where $I_p(\Po)$ is the polar moment of area of $\Po$ defined by
\beq
  I_p(\Po) = \int_{\Po} \left(x^2+y^2\right) \dd A
\eeq
whereby it is again assumed that the centroid of $\Po$ is located at the origin $O$ of the $x,y$-coordinate system, and the vertex $X_0$ lies on the positive $x$-axis (see \eqref{Eq:X_k}).
Now we rotate $\Po$ about $O$ with angle $-\alpha$ so that the side $\overline{X_0X_1}$ is parallel to the $y$-axis.
Using polar coordinates $\ph$ and $\rh$ we have $\rh^2 = x^2 + y^2$ and $\dd A = \rh\,\dd\ph\,\dd\rh$, we get
\beq
  I_p(\Po)
= \int_{\ph=0}^{2\pi}\int_{\rh=0}^{R(\ph)} \rh^3\,\dd\rh\,\dd\ph 
\eeq
where $R(\ph)$ is the equation of $\partial\:\!\Po$, and further
\begin{align*}
  I_p(\Po)
= {} & 2n \int_{\ph=0}^{\pi/n}\dd\ph \int_{\rh=0}^{R(\ph)}\rh^3\,\dd\rh
= \frac{n}{2} \int_{\ph=0}^{\pi/n}R(\ph)^4\,\dd\ph\db\\[0.1cm]
= {} & \frac{nr^4}{2}\cos^4(\pi/n) \int_{\ph=0}^{\pi/n}\left(1+\tan^2\ph\right)^2\dd\ph\db\\[0.1cm]
= {} & \frac{nr^4}{2}\cos^4(\pi/n) \int_{\ph=0}^{\pi/n}\frac{\dd\ph}{\cos^4\ph}\,.
\end{align*}
With Eq.\ 2.526.12 in \textcite[Vol.\ 1, p.\ 173]{Gradstein&Ryshik-engl} we get
\begin{align} \label{Eq:I_p}
  I_p(\Po)
= {} & \frac{nr^4}{2}\cos^4(\pi/n) \left(\frac{\sin(\pi/n)}{3\cos^3(\pi/n)} + \frac{2}{3}\tan(\pi/n)\right)
  \nonumber\db\\[0.1cm]   
= {} & \frac{nr^4}{6}\cos^4(\pi/n)\, \frac{\sin(\pi/n)+2\sin(\pi/n)\cos^2(\pi/n)}{\cos^3(\pi/n)}
  \nonumber\db\\[0.1cm]
= {} & \frac{nr^4}{6}\cos(\pi/n)\sin(\pi/n) \left[1+2\cos^2(\pi/n)\right]\nonumber\db\\[0.1cm]
= {} & \frac{nr^4}{6}\cos(\pi/n)\sin(\pi/n) \left[2+\cos(2\pi/n)\right].
\end{align}
Now from \eqref{Eq:Problem_XXVII} with the area
\beqn \label{Eq:A}
  A(\Po) = nr^2\cos(\pi/n)\sin(\pi/n)
\eeqn
(see also \eqref{Eq:L_and_A}) it follows that
\beq
  M_2(\Po)
= \frac{r^2}{3} \left(2 + \cos\frac{2\pi}{n}\right). \qedhere
\eeq
\end{proof}

\begin{remark}
We get
\beq
  \lim_{n\rightarrow\infty}M_2(\Po)
= \frac{r^2}{3}\, (2+1)
= r^2  
\eeq
which is the result for the circle with radius $r$ (see Table \ref{Ta:M_m(C_r)}). 
\end{remark}

\begin{cor}
The variance of the distance $\D_{n,\,r}$ between two random points in a regular polygon is given by
\beq
  \Var[\D_{n,\,r}]
= \E\big[\D_{n,\,r}^2\big] - \E[\D_{n,\,r}]^2
= M_2(\Po) - M_1(\Po)^2
\eeq
with $M_1(\Po)$ according to Corollary \ref{Cor:MD}, and $M_2(\Po)$ according to Theorem \ref{Thm:M_2(P_(n,r))}.
\end{cor}

Fig.\ \ref{Fig:Diagram_Variance} shows how the variances converge against the variance for the circle $\Cr$ with radius $r$ (see \eqref{Eq:Variance_Circle}).

\begin{SCfigure}[0.38][h]
  \includegraphics[width=0.6\textwidth]{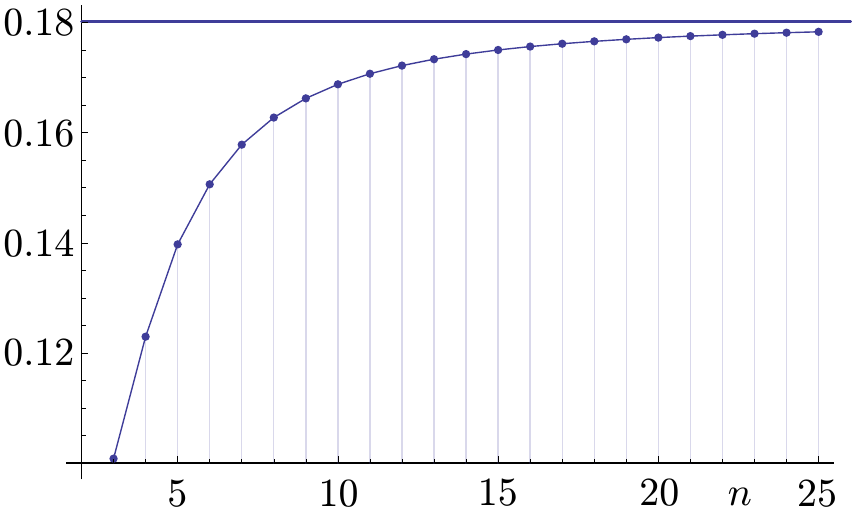}
  \caption{$\Var[\D_{n,\:\!r}]/r^2$, $n=3,4,\ldots,25$, and $\Var[\D_r]/r^2$}
  \label{Fig:Diagram_Variance}
\end{SCfigure}
% !TeX root = MDRP_0.tex

\section{Fourth moment}
\label{Sec:Fourth_moment}

\begin{thm} \label{Thm:M_4(P_(n,r))}
The fourth moment of the point distance in $\Po$, $r = 3,4,\ldots$, is given by
\beq
  M_4(\Po)
= \frac{r^4}{90} \left(77 + 64\cos\frac{2\pi}{n} + 9\cos\frac{4\pi}{n}\right).
\eeq
\end{thm}

\begin{proof}
From \eqref{Eq:S_m-T_m-a} together with the second formula in \eqref{Eq:S_m-T_m-b} we get
\beq
  M_4(\Po)
= \frac{S_7(\Po)}{21A(\Po)^2}\,.   
\eeq
According to \textcite[Eq.\ (16)]{Voss82} we can compute $S_7$ with
\beq
  S_7(\Po) = 42\left[A(\Po)I_p^{(4)}(\Po) + I_p(\Po)^2 + 2\left(I_x(\Po)^2+I_y(\Po)^2\right)\right]
\eeq
where
\begin{gather*}
  I_p^{(4)}(\Po)
= \int_{\Po} \left(x^2+y^2\right)^2 \dd A
= \frac{1}{6} \int_0^{2\pi} R(\ph)^6\, \dd\ph\,,\\[0.15cm]
  I_x(\Po)
= \int_{\Po} y^2\, \dd A\,,\qquad
  I_y(\Po)
= \int_{\Po} x^2\, \dd A
\end{gather*}
whereby the centroid of $\Po$ is located at the origin $O$ of the $x,y$-coordinate system and the axes are the principal axes of intertia. ($I_x$ und $I_y$ are the second moments of area with respect to the $x$-axis and the $y$-axis, respectively.)
So we have
\beqn \label{Eq:M_4(P_(n,r))}
  M_4(\Po)
= \frac{2I_p^{(4)}(\Po)}{A(\Po)} + \frac{2I_p(\Po)^2}{A(\Po)^2}
  + \frac{4\left(I_x(\Po)^2+I_y(\Po)^2\right)}{A(\Po)^2}\,. 
\eeqn
We get
\begin{align*}
  I_p^{(4)}(\Po)
= {} & \frac{n}{3} \int_0^{\pi/n}R(\ph)^6\,\dd\ph
= \frac{nr^6}{3}\cos^6(\pi/n) \int_0^{\pi/n}\left(1+\tan^2\ph\right)^3\dd\ph\db\\[0.1cm]
= {} & \frac{nr^6}{3}\cos^6(\pi/n) \int_0^{\pi/n}\frac{\dd\ph}{\cos^6\ph}\,,
\end{align*}
and, with Eq.\ 2.526.14 in \textcite[Vol.\ 1, p.\ 174]{Gradstein&Ryshik-engl},
\begin{align*}
  I_p^{(4)}(\Po)
= {} & \frac{nr^6}{3}\cos^6(\pi/n) \left(\frac{\sin(\pi/n)}{5\cos^5(\pi/n)} + \frac{4}{15}\tan^3(\pi/n)
  + \frac{4}{5}\tan(\pi/n)\right)\db\\[0.1cm]   
= {} & \frac{nr^6}{45}\cos(\pi/n)\sin(\pi/n) \left[3 + 4\sin^2(\pi/n)\cos^2(\pi/n)
  + 12\cos^4(\pi/n)\right]\db\\[0.1cm]   
= {} & \frac{nr^6}{45}\cos(\pi/n)\sin(\pi/n) \left[8 + 6\cos(2\pi/n) + \cos(4\pi/n)\right],
\end{align*}
and, with the area \eqref{Eq:A},
\beqn \label{Eq:Term_1}
  \frac{2I_p^{(4)}(\Po)}{A(\Po)}
= \frac{2r^4}{45} \left[8 + 6\,\cos(2\pi/n) + \cos(4\pi/n)\right]. 
\eeqn
With \eqref{Eq:I_p} we get
\begin{align} \label{Eq:Term_2}
  \frac{2I_p(\Po)^2}{A(\Po)^2}
= {} & 2\,\frac{\left[\frac{1}{6}nr^4\cos(\pi/n)\sin(\pi/n)\, \left(2+\cos(2\pi/n)\right)\right]^2}
	{(nr^2\cos(\pi/n)\sin(\pi/n))^2}
= \frac{r^4}{18} \left[2 + \cos(2\pi/n)\right]^2\nonumber\db\\[0.05cm]
= {} & \frac{r^4}{36} \left[9 + 8\cos(2\pi/n) + \cos(4\pi/n)\right].	  
\end{align}
Now we compute $I_x(\Po)$ and $I_y(\Po)$.
Since the axes of the coordinate system are principal axes of inertia and the origin is located in the centroid of $\Po$, the second area moments $I_x(\Po)$ and $I_y(\Po)$ take extreme values compared to all other second area moments with respect to axes through the centroid \autocite[p.\ 165]{Goeldner&Holzweissig}.
This means that one of the two values $I_x(\Po)$ and $I_y(\Po)$ is the maximum and one is the minimum value.
An axis of symmetry and the axis perpendicular to it are always principal axes of inertia \autocite[p.\ 165]{Goeldner&Holzweissig}.
Each polygon $\Po$, $n = 3,4,\ldots$, has more than two symmetry axes that are not perpendicular to each other.
Therefore it follows that the second area moment is equal for all axes through the centroid.
Thus with $I_x(\Po) + I_y(\Po) = I_p(\Po)$ we have found
\beq
  I_x(\Po) = I_y(\Po)
= \frac{1}{2}\,I_p(\Po)\,,
\eeq
hence
\beq
  \frac{4\left(I_x(\Po)^2+I_y(\Po)^2\right)}{A(\Po)^2}
= \frac{2I_p(\Po)^2}{A(\Po)^2}\,,    
\eeq
and \eqref{Eq:M_4(P_(n,r))} becomes
\beq
  M_4(\Po)
= \frac{2I_p^{(4)}(\Po)}{A(\Po)} + \frac{4I_p(\Po)^2}{A(\Po)^2} 
\eeq
(cp. Eq.\ (17) in \autocite{Voss82}).
So, with \eqref{Eq:Term_1} and \eqref{Eq:Term_2} we finally get
\begin{align*}
  M_4(\Po)
= {} & \frac{2r^4}{45} \left[8 + 6\,\cos(2\pi/n) + \cos(4\pi/n)\right]
  + \frac{r^4}{18} \left[9 + 8\cos(2\pi/n) + \cos(4\pi/n)\right]\\[0.05cm]
= {} & \frac{r^4}{90} \left(77 + 64\cos\frac{2\pi}{n} + 9\cos\frac{4\pi}{n}\right). \qedhere    
\end{align*}
\end{proof}

\begin{remark}
From Theorem \ref{Thm:M_4(P_(n,r))} we get
\beq
  \lim_{n\rightarrow\infty}M_4(\Po)
= \frac{r^4}{90}\, (77+64+9)
= \frac{5}{3}\, r^4  
\eeq
which is the result for the circle with radius $r$ (see Table \ref{Ta:M_m(C_r)}). 
\end{remark}

\begin{remark}
With the side length $\ell_1 = 2r\sin(\pi/n)$ and the area $A(\Po)$ (see \eqref{Eq:A}), the second area moment $I_x(\Po)$ with respect to an axis through the centroid of $\Po$ can be written in the form
\beq
  I_x(\Po) = \frac{A(\Po)}{12} \left(3r^2-\frac{\ell_1^2}{2}\right).
\eeq
This result can already be found in \textcite[p.\ 249]{Huette1872}.
The independence of $I_x(\Po)$ from the direction of the $x$-axis through the centroid is also discussed in \textcite{anderstood&ccorn}.
\end{remark}
% !TeX root = MDRP_0.tex

\section{Circle}
\label{Sec:Circle}

Now, we consider the moments $M_m(\Cr)$, $m = 1,2,\ldots$, of the distance $\D_r$ between two random points in a circle $\Cr$ with radius $r$.

\textcite{Czuber1884} found
\beqn \label{Eq:M_1(C_r)-M_2(C_r)}
  M_1(\Cr) = \frac{128}{45\pi}\,r\quad \mbox{(p.\ 197)}\,,\qquad
  M_2(\Cr) = r^2\quad \mbox{(p.\ 217)}\,.  
\eeqn
A numerical approximation of $M_1(\Cr)$ is given by
\beqn \label{Eq:M_1(C_r)}
  M_1(\Cr) \approx 0.90541478736722679904\,r\,.
\eeqn
A general formula for $M_m(\Cr)$ can be found in \textcite[p.\ 207, Theorem 2.3.15]{Mathai}. 
Here we calculate the moments using \eqref{Eq:M_m}, with $\Cr$ instead of $\Po$, and $2r$ instead of $d$, in the form
\beq
  M_m(\Cr) = \E[\D_r^m]
= \frac{2L(\partial\:\!\Cr)}{(m+2)A(\Cr)^2} \int_0^{2r} x^{m+2} \left(1-F_r(x)\right) \dd x    
\eeq
where $F_r(x)$ is the chord length distribution function of $\Cr$.
One easily finds
\beq
  F_r(x) = 1 - \sqrt{1-\left(\frac{x}{2r}\right)^2}
\eeq
(see also \textcite{Piefke}), hence
\begin{align*}
  M_m(\Cr)
= \frac{4\pi r}{(m+2)\pi^2 r^4} \int_0^{2r} x^{m+2}\, \sqrt{1-\left(\frac{x}{2r}\right)^2}\, \dd x\,.  
\end{align*}
The substitution
\beq
  t = \left(\frac{x}{2r}\right)^2
\eeq
yields
\beq
  M_m(\Cr)
= \frac{2^{m+4}\,r^m}{(m+2)\pi} \int_0^1 t^{(m+1)/2}\,(1-t)^{1/2}\, \dd t\,.  
\eeq
The last integral is the Euler beta function $B((m+3)/2,3/2)$.
Using the relationship to the gamma function $\Ga$,
\beq
  B(x,y) = \frac{\Ga(x)\Ga(y)}{\Ga(x+y)}\,,
\eeq 
and $\Ga(1/2) = \sqrt{\pi}$ one finds
\beqn \label{Eq:M_m(C_r)-2}
  M_m(\Cr)
= \frac{2^{m+4}\,r^m}{\sqrt{\pi}\,(m+2)(m+4)}\,
  \frac{\Ga\!\left(\frac{m+3}{2}\right)}{\Ga\!\left(\frac{m}{2}+2\right)}\,.  
\eeqn
This is Mathai's result.
Special values of \eqref{Eq:M_m(C_r)-2} are to be found in Table \ref{Ta:M_m(C_r)}\,.

\begin{table}[h]
\caption{Special values of $M_m(\Cr)$}
\label{Ta:M_m(C_r)}
\begin{center}
\begin{tabular}{|c|cccccccc|}
\hline
\rule{0pt}{12pt}
$m$ & 1 & 2 & 3 & 4 & 5 & 6 & 7 & 8\\[2pt] \hline
\rule{0pt}{20pt}
$M_m(\Cr)$ & $\dfrac{128}{45\pi}\,r$ & $r^2$ & $\dfrac{2048}{525\pi}\,r^3$ & $\dfrac{5}{3}\,r^4$ &
$\dfrac{16384}{2205\pi}\,r^5$ & $\dfrac{7}{2}\,r^6$ & $\dfrac{524288}{31185\pi}\,r^7$ &
$\dfrac{42}{5}\,r^8$\\[10pt] \hline
\rule{0pt}{12pt}
Eq. & \eqref{Eq:M_1(C_r)-M_2(C_r)} & \eqref{Eq:M_1(C_r)-M_2(C_r)} & & & & & &\\[2pt] \hline
\end{tabular}
\end{center} 
\end{table}
For the variance we have
\beqn \label{Eq:Variance_Circle}
\begin{aligned}
  \Var[\D_r]
= {} & \E\big[\D_r^2\big] - \E[\D_r]^2
= M_2(\Cr) - M_1(\Cr)^2
= \left[1-\left(\frac{128}{45\pi}\right)^{\!2}\right]r^2\\
\approx {} & 0.18022406281675948280\,r^2\,.   
\end{aligned}
\eeqn
% !TeX root = MDRP_0.tex

\section{Chord power integrals}
\label{Sec:CPI}

From \eqref{Eq:S_m-T_m-a} with the second formula in \eqref{Eq:S_m-T_m-b} one gets
\beq
  S_m(\Po)
= \frac{m(m-1)}{2}\,T_{m-3}(\Po)
= \frac{m(m-1)}{2}\,\AP^2\,M_{m-3}(\Po)\,,\quad m = 2,3,4,\ldots,   
\eeq
(see also \textcite[p.\ 47]{Santalo}).
Thus it is easy to compute chord power integrals $S_m(\Po)$ with the results of the present paper.
One obtains e.g.\ for the second chord power integral of $\mathcal{P}_{6,\:\!r}$
\begin{align*}
  S_2(\mathcal{P}_{6,\:\!r})
= A(\mathcal{P}_{6,\:\!r})^2\,M_{-1}(\mathcal{P}_{6,\:\!r})
= {} & \frac{27r^4}{4}
  \left(-\frac{16}{9} + \frac{8}{3\sqrt{3}} + \frac{22\ln 3}{9}- \frac{4\ln\left(2+\sqrt{3}\right)}{9}\right)
  \frac{1}{r}\\[0.02cm] 
= {} & \frac{27r^3}{4}
  \left(-\frac{16}{9} + \frac{8}{3\sqrt{3}} + \frac{22\ln 3}{9}- \frac{4\ln\left(2+\sqrt{3}\right)}{9}\right)
  \\[0.02cm]
\approx {} & 12.568534\,r^3  
\end{align*}
(cf.\ \textcite{Heinrich2009}). 

% !TeX root = MDRP_0.tex

%\section*{Ideen}
%\addcontentsline{toc}{section}{Ideen}
%
%\begin{itemize}[leftmargin=0.5cm]
%\setlength{\itemsep}{-1pt}
%%%%%%%%%%%
%\item bezogene Dichtefunktionen $F(x) = F(x/r)$, $f(x) = F'(x) = F'(x/r) = f(x/r)/r \quad\Longrightarrow\quad f(x/r) = r \times f(x)$
%%%%%%%%%%%
%\item Chord20g.nb: Momente des Punktabstandes: numerisch und analytisch\\
%Chord21c.nb: Punktabstandsdichtefunktion und numerische Momente fuer regul\"ares Polygon\\
%Chord24a.nb: Sehnenlängedichtefunktion und numerische Sehnenl\"angenmomente f\"ur regul\"ares Polygon -- neue, klarere Darstellung\\
%Chord25.nb: Closed form expressions\\
%Chord26.nb: Sehnenlaengenverteilungs- und -dichtefunktion\\
%Chord27.nb: Sehnenlaengendichtefunktion und Moments der Sehnenlaenge
%\end{itemize}

\addcontentsline{toc}{section}{References}
\printbibliography[title={References}]

\bigskip
{\bf Uwe Bäsel}, Leipzig University of Applied Sciences (HTWK Leipzig), Faculty of Engineering, PF 30\,11\,66, 04251 Leipzig, Germany, \texttt{uwe.baesel@htwk-leipzig.de}
  
\end{document}